\documentclass[a4paper,11pt]{amsart}

\usepackage{mathpazo}
\usepackage{amssymb}
\usepackage[pagebackref,
    ,pdfborder={0 0 0}
    ,urlcolor=black,a4paper,hypertexnames=false]{hyperref}
\hypersetup{pdfauthor={Clara L\"oh, Cristina Pagliantini, Sebastian Waeber},pdftitle=Cubical simplicial volume of 3-manifolds}

\usepackage{pgf}
\usepackage{tikz}
\usetikzlibrary{decorations.pathmorphing}
\usepackage[all]{xy}
\SelectTips{cm}{}

\newtheorem{thm}{Theorem}[section]
\newtheorem{prop}[thm]{Proposition}
\newtheorem{lem}[thm]{Lemma}
\newtheorem{cor}[thm]{Corollary}

\newtheorem{question}[thm]{Question}

\theoremstyle{remark}
\newtheorem{rem}[thm]{Remark}
\newtheorem{exa}[thm]{Example}

\theoremstyle{definition}
\newtheorem{defi}[thm]{Definition}

\usepackage{amsfonts}

\newcommand{\Q}{\mathbb{Q}}
\newcommand{\R}{\mathbb{R}}
\newcommand{\N}{\mathbb{N}}

\newcommand{\bb}{\partial}

\DeclareMathOperator{\smear}{smear}
\DeclareMathOperator{\im}{im}

\DeclareMathOperator{\Isom}{Isom}
\DeclareMathOperator{\vol}{vol}
\DeclareMathOperator{\sgn}{sgn}
\DeclareMathOperator{\id}{id}

\DeclareMathOperator{\map}{map}

\def\epsilon{\varepsilon}
\def\hyp{\mathbb{H}}

\DeclareMathOperator{\diam}{diam}

\DeclareMathOperator{\Top}{{\sf Top}}
\DeclareMathOperator{\Topp}{{\sf Top}^{2}}
\DeclareMathOperator{\Toppr}{{\sf Top}_{{\sf p}}}

\DeclareMathOperator{\Ob}{Ob}

\DeclareMathOperator{\nCh}{{\sf Ch}_{\R}^{{\sf n}}}
\DeclareMathOperator{\snCh}{{\sf Ch}_{\R}^{{\sf sn}}}
\DeclareMathOperator{\pnCh}{{\sf Ch}_{\R}^{{\sf sn},\infty}}

\def\args{\;\cdot\;}

\def\fa#1{%
  \forall_{#1}\;\;\;}

\def\sv#1{%
  \| #1 \|}
\def\svlf#1{%
  \| #1 \|_{\mathrm{lf}}}
\def\gsv#1{%
  \| #1 \|}
\def\gsvlf#1{%
  \| #1 \|_{\mathrm{lf}}}
\def\qsv#1{%
  \| #1 \|^{\square}}
\def\qsvlf#1{%
  \| #1 \|^{\square}_{\mathrm{lf}}}

\def\glone#1{%
  \| #1 \|}
\def\gclone#1{%
  | #1 |}
\def\qclone#1{%
  | #1 |_1^\square}
\def\clone#1{%
  | #1 |_1
}
\def\cdlone#1{%
  | #1 |_1^\triangle%
}

\def\fclr#1{%
  [#1]_\R}
\def\qfclr#1{%
  [#1]_\R^\square}
\def\fclrlf#1{%
  [#1]_\R^{\mathrm{lf}}}

\usepackage{color}
\usepackage{pdfcolmk}

\def\draftinfo{}

\author{Clara L\"oh}
\address{Fakult\"at f\"ur Mathematik\\
         Universit\"at Regensburg\\
         93040~Regensburg\\
         }
\email{clara.loeh@mathematik.uni-r.de}
\author{Cristina Pagliantini}
\address{Department Mathematik\\
         ETH Zentrum\\
         8092~Z\"urich\\
        }
\email{cristina.pagliantini@math.ethz.ch}
\author{Sebastian Waeber}
\address{Fakult\"at f\"ur Mathematik\\
         Universit\"at Regensburg\\
         93040~Regensburg\\
         }
\email{sebastian.waeber@stud.uni-r.de}
\title{Cubical simplicial volume of $3$-manifolds}
\date{\today.\ \copyright{\ C.~L\"oh, C.~Pagliantini, S.~Waeber 2015}. 
    This work was supported by the CRC~1085 \emph{Higher Invariants} 
    (Universit\"at Regensburg, funded by the DFG) and is related to the 
    Master thesis project of the third author. Moreover, C.L.\ is grateful 
    to the FIM at ETH Z\"urich for its hospitality. C.P.\ was supported by Swiss National
Science Foundation project 144373.\\
    MSC~2010 classification: 55N10, 57M27, 57N65 \draftinfo
}
\begin{document}

\begin{abstract}
  We prove that cubical simplicial volume of oriented closed 
  $3$-manifolds is equal to one fifth of ordinary simplicial volume.
\end{abstract}

\maketitle

\section{Introduction}

Simplicial volumes describe how difficult it is to represent the 
fundamental class of a given manifold in terms of cycles. Classically, 
simplicial volume was introduced by Gromov in terms of singular homology 
(defined via singular simplices)~\cite{vbc,mapsimvol}. Similarly, cubical 
simplicial volume is defined in terms of cubical singular homology. 

Gromov asked whether simplicial volume and cubical simplicial volume
are proportional in every dimension~\cite[5.40]{gromov}. This clearly
holds in dimension~$0$ and~$1$, and is known to be true for
surfaces~\cite{loehplankl}.

We will prove that cubical simplicial volume  $\qsv \cdot$ and ordinary simplicial volume $\sv \cdot$ 
are proportional in dimension~$3$:

\begin{thm}\label{mainthm}
  Let $M$ be an oriented closed $3$-manifold. Then 
  \[ \qsv M = \frac15 \cdot \sv M. 
  \]
\end{thm}

In higher dimensions, the corresponding question is still an open problem.

As for ordinary simplicial volume, the computation of cubical
simplicial volume in dimension~$3$ is based on decomposing the
manifold in question into hyperbolic and Seifert fibred
pieces. Therefore, the main steps of the proof are as follows:
\begin{itemize}
  \item We establish the general lower bound~$\qsv M \geq 1/5 \cdot \sv M$ 
    by subdividing cubes into five simplices 
    (Section~\ref{subsec:subdivision}).
  \item We show that Seifert fibred pieces have cubical simplicial 
    volume equal to~$0$ (Section~\ref{sec:seifert}).
  \item We show that hyperbolic pieces have cubical simplicial volume
    equal to one fifth of ordinary simplicial volume
    (Section~\ref{sec:hyp}).
  \item We prove a suitable (sub-)additivity for cubical simplicial volume 
    under gluings along tori (Section~\ref{sec:glue}).
\end{itemize}

More precisely, we will prove more general versions of all these steps: 

We develop a framework that allows for a convenient translation from
ordinary simplicial volume to other, generalised, simplicial volumes,
formulated in terms of so-called normed models of the singular chain
complex (see Definition~\ref{def:gnsch}). All normed models lead to
simplicial volumes that only differ by a multiplicative gap
(Proposition~\ref{prop:equivalence}). In particular, we discuss the
inheritance of vanishing of generalised simplicial volumes, which
yields the vanishing for Seifert fibred pieces as a special case
(Proposition~\ref{prop:vanishing},
Corollary~\ref{cor:seifert}). Furthermore, we show that simplicial
volumes associated with so-called geometric normed models satisfy a
sub-additivity with respect to gluings along UBC~boundaries
(Theorem~\ref{thm:glue}); this includes, in particular, manifolds
whose boundary components are tori. Hence, for all geometric normed
models we obtain an upper bound for the corresponding generalised
simplicial volume of $3$-manifolds in terms of their hyperbolic pieces
(Theorem~\ref{mainthmgen}).

Because there exist cubical normed models
(Proposition~\ref{prop:cubex}), we can apply these observations to
cubical simplicial volume. In addition, we investigate more efficient
cubical normed models in dimension~$3$, which give the desired lower
bounds in dimension~$3$ (Corollary~\ref{cor:3lowerbound}). 

Concerning the hyperbolic pieces, we calculate cubical simplicial
volume for oriented hyperbolic manifolds of finite volume in
dimension~$3$ (Corollary~\ref{cor:hyp3}) and give general, geometric
lower and upper bounds for the proportionality constant in all
dimensions (Section~\ref{sec:hyp}).

\subsection*{Organisation of this article}

In Section~\ref{sec:sv}, we recall the definition of classical
simplicial volume and cubical simplicial volume. The language of
normed models is developed in Section~\ref{sec:models} and cubical
models are studied in
Section~\ref{sec:cubicalmodels}. Section~\ref{sec:seifert} treats the
Seifert case and Section~\ref{sec:hyp} treats the hyperbolic case. In
Section~\ref{sec:glue}, we prove sub-additivity of generalised
simplicial volumes under certain gluings, and Section~\ref{sec:proof}
contains the proof of Theorem~\ref{mainthm} and its generalisation to
geometric normed models.

\subsection*{Acknowledgements}

We are very grateful to Michelle Bucher and Roberto Frigerio for
numerous helpful discussions.

\section{Simplicial volume and cubical simplicial volume}\label{sec:sv}

We start with a review of Gromov's simplicial
volume~\cite{vbc,mapsimvol} and some basic properties, as well as the
definition of cubical simplicial volume.

\subsection{Classical simplicial volume}

We denote the singular chain complex with $\R$-coefficients, based on
singular simplices, by~$C_*(\args;\R)$ and the corresponding singular
homology by~$H_*(\args;\R)$. 

Let $M$ be an oriented compact $n$-manifold (with possibly
empty boundary~$\partial M$). The corresponding $\R$-fundamental class
in~$H_n(M,\partial M;\R)$ is denoted by~$\fclr{M,\partial M}$. If
$\partial M = \emptyset$, then we write~$\fclr M$ instead.

\begin{defi}[simplicial volume, absolute case]
  The \emph{simplicial volume} of an oriented closed $n$-manifold
  is defined as
  \[ \sv M := \inf \bigl\{ \clone c 
                   \bigm| c \in C_n(M;\R), \partial c = 0, [c] = \fclr M 
                          \in H_n(M;\R) 
                    \bigr\},
                    \]
  where $\clone{\cdot}$ denotes the $\ell^1$-norm on~$C_n(M;\R)$ with respect 
  to the basis of singular $n$-simplices in~$M$.
\end{defi}

In general, the explicit computation of simplicial volume is rather
hard. Known results include the calculation of simplicial volume of
hyperbolic manifolds~\cite{vbc,thurston} and for manifolds that are
locally isometric to the product of two hyperbolic
planes~\cite{bucher}.

\subsection{Additivity of simplicial volume}

In dimension~$3$, simplicial volume can be computed by
cutting the manifold into smaller, well understood, pieces. To this
end, it is necessary to consider the simplicial volume of manifolds
with boundary and of open manifolds and to prove additivity under 
certain gluings. We now recall this calculation in more detail:

\begin{defi}[simplicial volume, relative case]\label{def:relsv}
The \emph{(relative) simplicial volume} of an oriented compact
manifold~$M$ is defined as
$$
\sv{M,\partial M}:=\inf\{\clone{c} \,|\, c\in C_n(M,\partial M;\R),\partial c =0 , [c]=[M,\partial M]_\R\}, 
$$
where $\clone{\cdot}$ denotes the quotient norm on~$C_*(M,\partial M;\R)$ 
induced by the $\ell^1$-norm on~$C_*(M;\R)$.
\end{defi}

If $M$ is an oriented manifold without boundary, the
fundamental class and the simplicial volume of $M$ admit analogous
definitions in the context of locally finite
homology~\cite{vbc,loehphd}.

\begin{defi}[simplicial volume, non-compact case]\label{def:lfsv}
  The \emph{(locally finite) simplicial volume} of an open $n$-manifold~$M$ is
  \begin{align*}
  \svlf{M} :=\inf\bigl\{\clone{c} 
                    \bigm| c\in C_n^{\mathrm{lf}}(M;\R), \partial c=0, [c]= 
                    \fclrlf M \in H_n^{\mathrm{lf}}(M;\R)
                    \bigr\}
  \in [0,\infty],
  \end{align*}
  where $\fclrlf M$ denotes the fundamental class of~$M$ in locally finite
  homology $H_n^{\mathrm{lf}}(M;\R)$ and $\clone \cdot$ denotes the $\ell^1$-norm on locally
  finite chains $C_n^{\mathrm{lf}}(M;\R)$.
\end{defi}

For example, an analysis of boundaries [or restrictions] of
relative [or locally finite] fundamental cycles gives the following
restrictions on simplicial volume of the
boundary~\cite{vbc,loehphd,loehl1}:

\begin{prop}[simplicial volume of the boundary]\label{prop:svboundary}
  Let $M$ be an oriented compact $n$-manifold and let $M^\circ := M
  \setminus \partial M$.
  \begin{enumerate}
    \item Then $\sv{\partial M} \leq (n+1) \cdot \sv M.$
    \item If $\svlf{M^\circ} < \infty$, then $\sv{\partial M} = 0$.
  \end{enumerate}
\end{prop}

\begin{rem}
Actually, in this situation, the following stronger inequality
$\sv{\partial M} \leq (n-1) \cdot \sv M$
holds~\cite[Lemma~2.3, Proposition~3.1]{BFP}. 
\end{rem}

However, in general it is \emph{not} true that $\svlf{M^\circ} =
\sv{M,\partial M}$, as can be seen by looking at~$M :=
   [0,1]$~\cite{vbc,loehphd}.

In the case with boundary, the exact value of simplicial volume was
computed for products of surfaces with the interval and for compact
$3$-manifolds obtained by adding $1$-handles to Seifert
manifolds~\cite{BFP}.  Further computations can be obtained by taking
connected sums or gluing along $\pi_1$-injective boundary components
with amenable fundamental groups, relying on the following result.

\begin{thm}[additivity of simplicial volume under amenable gluings]\label{thm:additivity}
  Let $k, n \in \N_{\geq 2}$, let $M_1,\dots,M_k$ be oriented compact
  $n$-manifolds, and suppose that the fundamental group of every
  boundary component of every~$M_j$ is amenable. Let $M$ be a manifold
  obtained by gluing $M_1,\dots, M_k$ along (some of) their boundary
  components. Then
  \[ \sv{M,\partial M}\leq \sv{M_1,\partial M_1}+\dots +\sv{M_k,\partial M_k}.
  \]
  In addition, if the gluings defining $M$ are compatible, then
  \[ \sv{M,\partial M}=\sv{M_1,\partial M_1}+\dots +\sv{M_k,\partial M_k}.\]
\end{thm}

Here, a gluing $f:S_1\rightarrow S_2$ of two boundary components
$S_i\subset \partial M_{j_i}$ is called \emph{compatible} if
$\pi_1(f)(K_1)=K_2$ where $K_i$ is the kernel of the map
$\pi_1(S_i)\rightarrow \pi_1(M_{j_i})$ induced by the inclusion.

This statement is originally due to Gromov~\cite[p.~58]{vbc}, and is also
discussed by Kuessner~\cite{kuessnerthesis}. A complete proof was
given in terms of a more algebraic framework~\cite[Theorem~3]{BBFIPP}.


In particular, additivity allows to compute simplicial volume of
closed $3$-mani\-folds with help of geometrization: In view of
additivity, the simplicial volume of a closed $3$-manifold equals the
sum of the simplicial volumes of its hyperbolic pieces~\cite{soma}.

As last step, one needs to calculate the simplicial volume of the
hyperbolic pieces~\cite{vbc,thurston,Fra,FP,FM,BBI}:

\begin{thm}[simplicial volume of hyperbolic manifolds]\label{thm:svhyp}
  Let $M$ be an oriented compact $n$-manifold 
 whose interior~$M^\circ$   admits a complete hyperbolic metric of finite volume. Then
  \[ \svlf{M^\circ} = \sv {M,\partial M} = \frac{\vol (M^\circ)}{v_n^\triangle}, 
  \]
  where $v_n^\triangle$ is the (finite!) volume of ideal, regular
  $n$-simplices in~$\overline{\hyp^n}$.
\end{thm}

Hence, in total we obtain~\cite{soma}:

\begin{thm}[simplicial volume of $3$-manifolds]\label{thm:sv3}
  Let $M$ be an oriented closed $3$-manifold. Then
  \[ \sv M = \sum_{j=1}^k \frac{\vol{(N_j^\circ)}}{v_3^\triangle}, 
  \]
  where $N_1, \dots, N_k$ are the hyperbolic pieces of~$M$ (see
  Theorem~\ref{thm:geom}).
\end{thm}

\subsection{Cubical simplicial volume}

We will now recall the definition of cubical singular
homology~\cite{masseyhom,eilenberg} and of cubical simplicial volume.
Cubical singular homology is defined in terms of standard cubes
instead of standard simplices: For $n\in \N$ let $\square^n :=
[0,1]^n$ be the \emph{standard $n$-cube}. If $X$ is a topological
space, then continuous maps of type~$\square^n \longrightarrow X$ are
called \emph{singular $n$-cubes of~$X$}. The geometric/combinatorial
boundary of~$\square^n$ consists of $2\cdot n$ cubical faces. 
More precisely, for~$j \in \{1,\dots,n\}$ and $i \in \{0,1\}$ we 
define the \emph{$(j,i)$-face of~$\square^n$} by
\begin{align*}
  \square^n_{(j,i)} \colon \square^{n-1}& \longrightarrow \square^n \\
    x & \longmapsto (x_1,\dots,x_{j-1}, i , x_{j}, \dots, x_{n-1}).
\end{align*}
Correspondingly, for a singular $n$-cube~$c \colon \square^n \longrightarrow X$ 
in a space~$X$ we define the cubical boundary by
\[ \partial c := 
   \sum_{j=1}^n (-1)^{j} \cdot (c \circ \square^n_{(j,0)} - c\circ \square^n_{(j,1)}).
\]
This leads to a chain complex~$Q_*(X;\R)$ of cubical singular chains with
$\R$-coefficients. A singular $n$-cube~$c$ is \emph{degenerate} if it is independent of one of
the coordinates, i.e., if there exists a~$j \in \{1,\dots,n\}$ such that 
for all~$t \in [0,1]$ and all~$x \in \square^{n-1}$ we have
\[ 
     c(x_1,\dots, x_{j-1}, t, x_j, \dots, x_{n-1}) 
   = c(x_1,\dots, x_{j-1}, 0, x_j, \dots, x_{n-1}).
\]
Dividing out the subcomplex~$D_*(X;\R)$ generated
by degenerate singular cubes leads to the \emph{cubical chain complex}
\[ C_*^\square(X;\R) := Q_*(X;\R) / D_*(X;\R)
\] 
and hence to \emph{cubical singular homology~$H_*^\square(X;\R)$}
(which admits a natural extension to a functor). Dividing
out degenerate singular cubes is necessary in order for cubical
singular homology of a point to be concentrated in degree~$0$.

We write~$\qclone{\cdot}$ for the norm on~$C_*^\square(\args;\R)$
induced from the $\ell^1$-norm on~$Q_*(\args;\R)$ with respect to the
basis consisting of singular cubes; notice that this norm coincides
with the $\ell^1$-norm given by the basis of all non-degenerate
singular cubes.

It is well known that there is a canonical natural (both in spaces and
coefficients) isomorphism~$H_*^\square \longrightarrow
H_*$~\cite[Theorem~V]{eilenberg} between cubical singular homology and
ordinary singular homology. However, in general this isomorphism is
\emph{not} isometric with respect to the corresponding
$\ell^1$-semi-norms.

If $M$ is an oriented closed manifold, we denote the
corresponding cubical $\R$-fundamental class by~$\qfclr M \in
H^\square_*(M;\R)$.

\begin{defi}[cubical simplicial volume, absolute case]
  The \emph{cubical simplicial volume} 
  of an oriented closed $n$-manifold~$M$  is defined as
  \[ \qsv M := \inf \bigl\{ \qclone c 
                   \bigm| c \in C^\square_n(M;\R), \partial c = 0, [c] = \qfclr M 
                          \in H_n^\square(M;\R) 
                    \bigr\}.
                    \]
\end{defi}

It is not hard to show that for every~$n \in \N$ there are
$T_n, V_n \in \R_{\geq 0}$ such that for all oriented closed
$n$-manifolds~$M$ we have
\[ T_n \cdot \sv M \leq \qsv M \leq V_n \cdot \sv M 
\]
(see Corollary~\ref{cor:svvsqsv} below). However, inheritance
results such as (sub-)ad\-di\-tivity cannot be derived directly from these
estimates.

For the divide and conquer approach to computing cubical simplicial
volume of $3$-manifolds we will need relative and locally finite
versions of cubical simplicial volume and corresponding
(sub-)additivity statements. While we could this literally in the same
way as in the simplicial case, we prefer to develop these tools in
slightly larger generality in the subsequent sections.

\section{Normed models of the singular chain complex}\label{sec:models}

We will now develop a framework that allows for a convenient
translation from ordinary simplicial volume to other simplicial
volumes, such as cubical simplicial volume.

\subsection{Basic terminology for normed models}

Let $\nCh$ be the category of normed $\R$-chain complexes, i.e., the
category of chain complexes whose chain modules are normed $\R$-vector
spaces, whose boundary maps are continuous $\R$-linear maps and whose
chain maps consist of $\R$-linear maps of norm at most~$1$.  Moreover,
we write $\Top$ for the category of topological spaces and continuous
maps.

\begin{defi}[normed models of the singular chain complex]\label{def:gnsch}
  \hfil
  \begin{itemize}
  \item
    A \emph{functorial normed chain complex} is a functor~$\Top
    \longrightarrow \nCh$.
  \item
    Let $F$, $F' \colon \Top \longrightarrow \nCh$ be functorial normed
    chain complexes. A \emph{natural continuous chain map~$F
      \Longrightarrow F'$} is a natural transformation~$\varphi \colon
    F \Longrightarrow F'$ viewed in the category of $\R$-chain
    complexes, where for every space~$X$ and $n \in \N$ the linear
    map~$\varphi_n^X \colon F_n(X) \longrightarrow F'_n(X)$ is
    continuous (but not necessarily of norm at most~$1$).
  \item
    Let $\varphi, \varphi' \colon F \Longrightarrow F'$ be natural
    continuous chain maps between functorial normed chain complexes.
    A \emph{natural normed chain homotopy $h \colon
      \varphi \simeq \varphi'$} is a family~$(h^X \colon \varphi^X
    \simeq \varphi'{}^X)_{X \in \Ob(\Top)}$ of $\R$-linear chain homotopies
    such that for all spaces~$X$ and all~$n \in \N$ the linear
    map~$h_n^X \colon F_n(X) \longrightarrow F'_{n+1}(X)$ is
    continuous.
  \item
    Natural continuous chain maps~$\varphi \colon F \Longrightarrow
    F'$ and $\psi \colon F' \Longrightarrow F$ are \emph{mutually
      inverse} if there exist natural normed chain homotopies $\psi
    \circ \varphi \simeq \id_F$ and $\varphi \circ \psi \simeq
    \id_{F'}$.  In this case, $\varphi$ and $\psi$ are \emph{natural
      normed chain homotopy equivalences}.
  \item
    A \emph{normed model of the singular chain complex} is a
    triple~$(F, \varphi, \psi)$, where $F \colon \Top \longrightarrow
    \nCh$ is a functorial normed chain complex, and where $\varphi
    \colon F \Longrightarrow C_*(\args;\R)$ and $\psi \colon
    C_*(\args;\R) \Longrightarrow F$ are mutually inverse natural
    normed chain homotopy equivalences. Here, $C_*(\args;\R)$ is 
    equipped with the $\ell^1$-norm.
  \end{itemize}
\end{defi}

For example, $(C_*(\args;\R), \id_{C_*(\args;\R)}, \id_{C_*(\args;\R)})$ is 
a normed model of the singular chain complex. 

\begin{exa}[ignoring degenerate singular simplices]\label{exa:degenerate}
  In analogy with the cubical case, we can also look at the singular 
  chain complex modulo degenerate simplices: For a topological space~$X$, 
  let $E_*(X;\R) \subset C_*(X;\R)$ be the subcomplex generated by 
  all degenerate singular simplices. Then we define
  \[ C_*^\triangle(\args;\R) := C_*(\args;\R) / E_*(\args;\R)
     \colon \Top \longrightarrow \nCh, 
  \]
  where we equip~$C_*^\triangle(\args;\R)$ with the
  norm~$\cdlone{\cdot}$ induced from the~$\ell^1$-norm (this coincides
  with the $\ell^1$-norm with respect to the basis of non-degenerated
  singular simplices).

  We will now construct a normed model of~$C_*(\args;\R)$
  on~$C_*^\triangle(\args;\R)$: Let $\psi \colon C_*(\args;\R)
  \Longrightarrow C_*^\triangle(\args;\R)$ be the natural continuous
  chain map given by the canonical projection. Conversely, we consider the map 
  \begin{align*}
    \varphi_n^X \colon C_n^\triangle(X;\R) 
    & \longrightarrow C_n(X;\R) \\
    [\sigma \colon \Delta^n \rightarrow X] 
    & \longmapsto \frac1{(n+1)!} \cdot \sum_{\pi \in \Isom(\Delta^n)} 
                  (-1)^{\sgn \pi} \cdot \sigma \circ \pi
  \end{align*}
where $\sgn \pi\in \{0,1\}$ encodes whether $\pi$ is orientation preserving or not.
It is   not hard to show that $\varphi_n^X$ defines a well-defined natural chain map that is continuous in every
  degree. Thus, we obtain a natural continuous chain map~$\varphi
  \colon C_*^\triangle(\args;\R) \Longrightarrow C_*(\args;\R)$.

  A standard argument (e.g., via acyclic models~\cite{eilenberg})
  shows that $\varphi$ and $\psi$ are mutually inverse natural normed
  chain homotopy equivalences. Hence, $(C_*^\triangle(\args;\R),
  \varphi,\psi)$ is a normed model of~$C_*(\args;\R)$. Notice that, by
  construction, $\|\varphi^X_n\| \leq 1$ and $\|\psi^X_n\| \leq 1$ for
  all spaces~$X$ and all~$n \in \N$. 
\end{exa}

Cubical normed models of the singular chain complex will be
constructed in Section~\ref{sec:cubicalmodels}.

\subsection{Extending normed models}

Normed models of~$C_*(\args;\R)$ also lead to corresponding notions
for manifolds with boundary and for open manifolds: Let $\Topp$ be the
category of pairs of topological spaces (and maps of pairs) and let
$\Toppr$ be the category of topological spaces and proper continuous
maps. Moreover, let $\snCh$ be the category of semi-normed $\R$-chain
complexes and let $\pnCh$ be the category of semi-normed $\R$-chain
complexes where also the value~$\infty$ is allowed for the norms. The
notions of functorial normed chain complexes etc.\ from
Definition~\ref{def:gnsch} easily generalise to these variations.

We first extend normed models to the relative case:

\begin{defi}[normed models, relative case]
  Let $(F,\varphi,\psi)$ be a normed model of~$C_*(\args;\R)$ with 
  norm~$\gclone{\cdot}^F$. 
  For a pair~$(X,A)$ of spaces with inclusion~$i \colon A \hookrightarrow X$ 
  we set
  \[ F(X,A) := F(X) \bigm/ F(i)(F(A)),
  \]
  which we endow with the quotient semi-norm $\gclone{\cdot}^F$ of~$F(X)$. 
  Hence, $F(\cdot, \cdot)$ defines a functor from $\Topp$ to $\snCh$.
  Moreover,
  $\varphi^X$ and $\psi^X$ induce well-defined natural continuous
  chain maps $\varphi^{(X,A)} \colon F(X,A) \longrightarrow
  C_*(X,A;\R)$ and $\psi^{(X,A)} \colon C_*(X,A;\R) \longrightarrow
  F(X,A)$ (which are also naturally mutually inverse chain homotopy
  equivalences through degree-wise continuous chain homotopies).
\end{defi}

\begin{exa}
Notice that for the trivial normed model~$(C_*(\args;\R), \id, \id)$
of~$C_*(\args;\R)$ this definition yields the same normed structure 
on the relative singular chain complex as the one used in the definition 
of relative simplicial volume (Definition~\ref{def:relsv}).
\end{exa}

\begin{defi}[locally finite normed models]
  Let $(F,\varphi,\psi)$ be a normed model of~$C_*(\args;\R)$. For a
  topological space~$X$ we let $K(X)$ denote the set of all compact
  subspaces of~$X$ (directed by inclusion). For~$X$ we then define
  \[ F^{\mathrm{lf}}(X) 
     := \varprojlim_{K \in K(X)} F(X, X \setminus K),
  \]
  where the inverse limit is taken with respect to the maps induced by
  the inclusions. For~$c = (c_K)_{K \in K(X)} \in F_n^{\mathrm{lf}}(X)$ we set
  \[ \gclone c ^F 
  := \inf_{\overline c \in A(c)} 
  \lim_{K \in K(X)} \gclone{\overline c_K}^F
  \in [0,\infty],
  \]
  where
  \begin{align*}
    A(c) := \bigl\{ (\overline c_K \in F_n(X))_{K \in K(X)}
    \bigm| &\;\text{$\overline c_K$ represents~$c_K$ 
      in $F_n(X,X \setminus K)$}
    \\
    &\;\text{and $(\overline c_K)_{K \in K(X)}$ is $\gclone{\cdot}^F$-Cauchy}  
    \bigr\}
  \end{align*}
  is the (possibly empty) set of all Cauchy families associated
  with~$c$.  This defines a (potentially infinite) semi-norm
  on~$F^{\mathrm{lf}}(X)$.  
  Hence, $F^{\mathrm{lf}}(\cdot)$ defines a functor from $\Toppr$  to $\pnCh$.
  Moreover, $\varphi^X$ and $\psi^X$ induce
  well-defined natural continuous chain maps~$\varphi^{\mathrm{lf},X}:F^{\mathrm{lf}}(X)
  \longrightarrow C^{\mathrm{lf}}_*(X;\R)$ and
  $\psi^{\mathrm{lf},X}:C^{\mathrm{lf}}_*(X;\R) \longrightarrow F^{\mathrm{lf}}(X)$ (which
  are also naturally mutually inverse chain homotopy equivalences
  through degree-wise continuous chain homotopies).
\end{defi}

Notice that we have 
\[ \gclone c ^F
\geq \sup_{K \in K(X)} \gclone{c_K}^F
\]
for all~$c \in F^{\mathrm{lf}}(X)$ and that equality holds for
the trivial model~$C_*(\args;\R)$ and all cubical models. It is not
clear whether equality holds for all normed models of~$C_*(\args;\R)$,
but the above definition seems to be more suitable for geometric
applications.

\begin{exa}\label{exa:lfalternative}
For the trivial normed model~$(C_*(\args;\R), \id, \id)$
of~$C_*(\args;\R)$ the above definition coincides with the locally finite
singular complex and yields the same normed structure on the locally
finite singular chain complex as the one used in the definition of
simplicial volume of open manifolds (Definition~\ref{def:lfsv}).
\end{exa}

\subsection{Generalised simplicial volume}

Normed models of the singular chain complex induce semi-norms on
singular homology:

\begin{defi}[induced semi-norm on homology]
  Let $(F,\varphi,\psi)$ be a normed model of~$C_*(\args;\R)$. For a 
  pair~$(X,A)$ of spaces we define
  \begin{align*}
    \glone{\cdot}^F \colon H_n(X,A;\R) 
    & \longrightarrow \R_{\geq 0}
    \\
    \alpha 
    & \longmapsto
    \inf \bigl\{ |c|^F 
    \bigm| \text{$c \in F_n(X,A)$, $\partial c =0$, 
      $[\varphi_n^{(X,A)}(c)] = \alpha$}
    \bigr\}.
  \end{align*}
  Similarly, for a topological space~$X$ we define
  \begin{align*}
    \glone{\cdot}^F \colon H^{\mathrm{lf}}_n(X;\R)
    & \longrightarrow [0,\infty]
    \\
    \alpha
    & \longmapsto
    \inf \bigl\{ |c|^F 
    \bigm| \text{$c \in F^{\mathrm{lf}}_n(X)$, $\partial c =0$, 
      $[\varphi_n^{\mathrm{lf},X}(c)] = \alpha$}
    \bigr\}.    
  \end{align*}
\end{defi}

In particular, we obtain corresponding simplicial volumes.

\begin{defi}[generalised simplicial volume]
  Let $(F,\varphi,\psi)$ be a normed model of~$C_*(\args;\R)$, 
  and let $M$ be an oriented compact 
  $n$-manifold. Then the \emph{$F$-simplicial volume of~$M$} is
  defined as
  \[ \gsv {M,\partial M}^F 
     := \glone{H_n(\psi_n^{(M,\partial M)})(\fclr {M,\partial M})}^F.
  \]
  If $\partial M = \emptyset$, we abbreviate~$\gsv M ^F := \gsv
  {M,\partial M}^F$. If $M$ is an oriented $n$-manifold 
  without boundary, then we define
  \[ \gsvlf{M}^F 
     := \glone{H_n(\psi_n^{\mathrm{lf},M})(\fclrlf M)}^F.
  \]
\end{defi}

One should note that the $F$-simplicial volume, in general, also
depends on~$\varphi$ and~$\psi$; however, in the cubical case, we will
add a normalisation condition that removes this ambiguity
(Remark~\ref{rem:normalisation}).

For simplicity, in the compact case, cycles in~$F_n(M,\partial M)$
that represent $H_n(\psi_n^{(M,\partial M)})(\fclr{M,\partial M})$ are
also called \emph{(relative) $F$-fundamental cycles}; in the
non-compact case, cycles in~$F_n^{\mathrm{lf}}(M)$ that represent
$H_n(\psi_n^{\mathrm{lf},M})(\fclrlf M)$ are also called \emph{locally
  finite $F$-fundamental cycles}.

It is not hard to show that all these simplicial volumes are
relative/proper homotopy invariants of the manifolds in question. At
this point it is essential that normed models map (proper) continuous
maps to chain maps of norm at most~$1$.

By the very definitions, all generalised norm chain complexes yield
equivalent simplicial volumes in the following sense:

\begin{prop}[equivalence of generalised simplicial volumes]\label{prop:equivalence}
  Let $(F,\varphi,\psi)$ be a normed model for~$C_*(\args;\R)$, and
  let $n \in \N$.
   \begin{enumerate}
    \item Then for all oriented compact $n$-manifolds~$M$ we have
      \begin{align*} 
        \gsv{M,\partial M}^F & \leq \|\psi_n^{\Delta^n}\| \cdot \sv{M,\partial M}
        , \\
        \sv{M,\partial M} & \leq \|\varphi_n^{(M,\partial M)}\| \cdot \gsv{M,\partial M}^F. 
      \end{align*}
    \item
      For all oriented $n$-manifolds~$M$ without boundary we have
      \begin{align*} 
        \gsvlf M ^F & \leq \|\psi_n^{\Delta^n}\| \cdot \svlf M
        , \\
        \svlf M & \leq \|\varphi_n^{\mathrm{lf},M}\| \cdot \gsvlf M^F.
      \end{align*}
  \end{enumerate}
\end{prop}
\begin{proof}
  \emph{Ad~1.}  The estimate $\sv{M,\partial M} \leq
  \|\varphi_n^{(M,\partial M)}\| \cdot \gsv{M,\partial M}^F$ easily
  follows from the definitions.  Conversely, let $c = \sum_{j=1}^k a_j
  \cdot \sigma_j \in C_n(M;\R)$ be a chain with~$\partial c \in
  C_{n-1}(\partial M;\R)$ that represents~$\fclr{M,\partial M}$. 
  Then
  \begin{align*}
    \bigl| \psi_n^{(M,\partial M)}[c] \bigr|^F
    & \leq \sum_{j=1}^k |a_j| 
           \cdot \bigl|\psi_n^M(\sigma_j \circ \id_{\Delta^n})\bigr|^F
    = \sum_{j=1}^k |a_j| 
           \cdot \bigl|F(\sigma_j) (\psi_n^{\Delta^n}(\id_{\Delta^n}))\bigr|^F
    \\
    & \leq \sum_{j=1}^k |a_j| \cdot 1 \cdot \| \psi_n^{\Delta^n}\| \cdot 1
    \\
    & \leq \|\psi_n^{\Delta^n}\| \cdot \clone c. 
  \end{align*}
  Hence, $\gsv{M,\partial M}^F \leq \|\psi_n^{\Delta^n}\| \cdot \sv{M,\partial M}$. 

  \emph{Ad~2.} This follows via the same arguments as in the first
  part, keeping in mind that continuous linear maps map Cauchy
  families to Cauchy families.
\end{proof}

\begin{exa}\label{exa:equivalence}
  For example, the normed model of~$C_*(\args;\R)$
  on~$C_*^\triangle(\args;\R)$ given in Example~\ref{exa:degenerate}
  leads to the ordinary simplicial volume because the corresponding
  chain homotopy equivalences all have norm at most~$1$.
\end{exa}

In particular, vanishing of generalised simplicial volumes is 
equivalent for all normed models of~$C_*(\args;\R)$. 
In contrast, it is not clear that sub-additivity properties 
of type~$\gsv M ^F \leq \gsv {N_1}^F + \gsv{N_2} ^F$ are also 
inherited in the same way.

The whole point of introducing the framework of normed models of the
singular chain complex and the discussion in Section~\ref{sec:glue} is
that we can in fact prove inheritance of sub-additivity in
sufficiently benign situations.

\subsection{The uniform boundary condition}

We now review the uniform boundary condition~\cite{mm}, which will
play an important role in our gluing constructions below.

\begin{defi}[UBC]\label{def:ubc}
  Let $q \in \N$.
  \begin{itemize}
    \item A normed chain complex~$(C,|\cdot|)$ satisfies the \emph{uniform boundary 
      condition in degree~$q$} (\emph{$q$-UBC}, for short) if there exists 
      a~$\kappa \in \R_{\geq 0}$ with the following property:
      For all~$z \in \im \partial_{q+1} \subset C_q$ there is 
      a~$b \in C_{q+1}$ with
      \[ \partial_{q+1} (b) = z 
         \quad\text{and}\quad
         |b| \leq \kappa \cdot |z|.
      \]
    \item Let $F \colon \Top \longrightarrow \nCh$ be a functorial normed 
      chain complex. A space~$X$ \emph{satisfies $q$-UBC with respect to~$F$} 
      if the normed chain complex~$F(X)$ satisfies $q$-UBC.
  \end{itemize}
\end{defi}

For example, the singular chain complex~$C_*(X;\R)$ satisfies the
uniform boundary condition in all degrees if $X$ is a connected
CW-complex with amenable fundamental group~\cite{mm}. This includes,
in particular, all tori and all simply connected spaces.

\begin{prop}[UBC inheritance]
  Let $(F,\varphi,\psi)$ be a normed model of the singular chain
  complex, let $q \in \N$, and let $X$ be a topological space. Then
  $X$ satisfies $q$-UBC with respect to the trivial
  model~$C_*(\args;\R)$ if and only if $X$ satisfies $q$-UBC with
  respect to~$F$.
\end{prop}
\begin{proof}
  This is a straightforward calculation: Because the situation
  basically is symmetric, we only show that if $C_*(X;\R)$ satisfies
  $q$-UBC, then also $F(X)$ satisfies $q$-UBC. 

  Let $\kappa \in\R_{\geq 0}$ be a constant witnessing that
  $C_*(X;\R)$ satisfies $q$-UBC and let $z \in
  \partial(F_{q+1}(X))$. Then $z' := \varphi_q^X(z)$ lies
  in~$\partial(C_{q+1}(X))$ and in view of the uniform boundary
  condition there exists a chain~$b' \in C_{q+1}(X;\R)$ with
  \[ \partial b' = z'
     \quad\text{and}\quad
     \clone{b'} \leq \kappa \cdot \clone{z'}.
  \]
  
  Let $h_* \colon \psi^X \circ \varphi^X \simeq \id_{F(X)}$ be a 
  chain homotopy that is continuous in every degree and let
  \[ b := \psi_{q+1}^X(b') - h_q(z)
     \in F_{q+1}(X). 
  \]
  Then 
  \[ \partial b = \psi_{q}^X (z') - \psi_q^X \circ \varphi_q^X(z) + z 
                - h_{q-1} (\partial (z))  
                = z
  \]
  and
  \[ |b|^F \leq \bigl(\| \psi_{q+1}^X \| \cdot \kappa \cdot \|\varphi_q^X\| 
                     + \| h_q\| \bigr)
                \cdot |z|^F.
  \]
  Hence, $F(X)$ satisfies $q$-UBC.
  
  Alternatively, one can also use the characterisation of the uniform
  boundary condition in terms of bounded cohomology~\cite{mm} to prove
  this inheritance statement (because normed models are set up in such
  a way that they yield naturally continuously isomorphic bounded
  cohomology theories).
\end{proof}

Hence, we can just say that a space satisfies $q$-UBC without
specifying a normed model of the singular chain complex (the
UBC constants, however, in general will depend on the chosen model).

\section{Cubical models of the singular chain complex}\label{sec:cubicalmodels}

We will now give general bounds between ordinary and cubical
simplicial volume, based on natural chain homotopy equivalences with
controlled norms. We formulate these results in terms of normed models
of the singular chain complex developed in Section~\ref{sec:models}.

\subsection{Cubical models}

We introduce the following shorthand:

\begin{defi}[cubical model]
  A \emph{cubical normed model of~$C_*(\args;\R)$} is a
  normed model~$(F,\varphi,\psi)$ of the singular chain complex
  with~$F = C_*^\square(\args;\R)$ that satisfies the normalisation
  \[ \varphi_0^{\square^0}(\id_{\square^0}) = (1 \mapsto 0) 
     \colon \Delta^0 \longrightarrow \square^0.
  \]
\end{defi}

This normalisation rules out unwanted scaling:

\begin{rem}\label{rem:normalisation}
  Let $(F,\varphi,\psi)$ be a cubical normed model of~$C_*(\args;\R)$.
  Then the acyclic models theorem~\cite{eilenberg}, the normalisation
  condition, and the local characterisation of fundamental classes
  show that for all oriented manifolds the cubical fundamental class
  and the ordinary fundamental class are mapped to each other by the
  given chain homotopy equivalences. Thus, the $F$-simplicial volume
  coincides with cubical simplicial volume (also for the
  straightforward adaptions to the relative and the non-compact case);
  so, the slightly sloppy term ``$F$-simplicial volume'' (without
  reference to~$\varphi$ or~$\psi$) is not ambiguous in this case.
\end{rem}

\subsection{Construction of cubical models}

As next step, we show that cubical models indeed exist. 

\begin{prop}\label{prop:cubex}
  There exist cubical normed models of~$C_*(\args;\R)$.
\end{prop}
\begin{proof}
  Via acyclic models, Eilenberg and Mac~Lane~\cite{eilenberg}
  constructed a natural chain map~$C_*^\square(\args;\R)
  \Longrightarrow C_*^\triangle(\args;\R)$ that satisfies the
  normalisation condition ($C_*^\triangle(\args;\R)$ was defined
  in Example~\ref{exa:equivalence}). In combination with the natural
  chain homotopy equivalence~$C_*^\triangle(\args;\R) \simeq
  C_*(\args;\R)$, this produces a natural chain
  map~$C_*^\square(\args;\R) \Longrightarrow C_*(\args;\R)$. Applying
  Lemma~\ref{lem:acyclicmodels} below then proves the claim.
\end{proof}

\begin{lem}\label{lem:acyclicmodels}
  Let $\varphi \colon C_*^\square(\args;\R) \Longrightarrow
  C_*(\args;\R)$ be a natural chain map with
  \[ \varphi_0^{\square^0}(\id_{\square^0}) = (1 \mapsto 0). 
  \]
  Then $\varphi$ is continuous in the sense of
  Definition~\ref{def:gnsch} and there exists a natural continuous
  chain map~$\psi \colon C_*(\args;\R) \Longrightarrow C_*^\square(\args;\R)$
  such that $(C_*^\square(\args;\R),\varphi,\psi)$ is a cubical normed
  model of~$C_*(\args;\R)$.
\end{lem}
\begin{proof}
  The acyclic models technique shows that there is a natural chain
  map~$\psi \colon C_*(\args;\R) \Longrightarrow
  C_*^\square(\args;\R)$ such that $\varphi \circ \psi$ and $\psi
  \circ \varphi$ are naturally chain homotopic to the
  identity~\cite{eilenberg}. 

  Because $C_*^\square(\args;\R)$ is basically generated
  by~$(\id_{\square^n})_{n \in \N}$ and naturality, the same arguments
  as in the proof of Proposition~\ref{prop:equivalence} show that
  $\varphi$ and $\psi$ are continuous in the sense of
  Definition~\ref{def:gnsch} and that also the corresponding natural
  chain homotopies are continuous. More precisely,  
  \[ \| \varphi_n^X \| \leq \|\varphi_n^{\square^n}\| 
     = \clone{\varphi_n^{\square^n}(\id_{\square^n})} 
     \quad\text{and}\quad
     \| \psi_n^X \| \leq \|\psi_n^{\Delta^n}\| 
     = \qclone{\psi_n^{\Delta^n}(\id_{\Delta^n})},
  \]
  for all spaces~$X$ and all~$n \in \N$.  
  Hence, $(C_*^\square(\args;\R),\varphi,\psi)$ is a cubical normed
  model of~$C_*(\args;\R)$.
\end{proof}

We will now give more explicit bounds:

\begin{rem}\label{rem:eilenbergbound}
 The construction of Eilenberg and Mac~Lane~\cite{eilenberg} produces
 a chain homotopy inverse~$\psi$ with $\|\psi_n^{\Delta^n}\| \leq 1$
 for all~$n \in \N$.
\end{rem}

Conversely, an explicit construction of a natural chain
map 
\[ \varphi' \colon C_*^\square(\args;\R) \Longrightarrow C_*^\triangle(\args;\R)
\] 
that satisfies the normalisation condition can be obtained by an
inductive triangulation procedure via cones:

  \begin{figure}
    \begin{center}
      \begin{tikzpicture}[thick]
        \fill[black!10] (-1.5,0) -- (1.5,0) -- (0,2) -- cycle;
        \foreach \i in {-1.5,-1.25,...,1.5}
        { \draw[black!30] (\i,0) -- (0,2);
        }
        \draw (-1.5,0) -- (1.5,0);
        \fill (0,2) circle (0.1);
        \draw (0,0) node[anchor=north] {$\sigma$};
        \draw (0,2) node[anchor=south] {$v$};
        \draw (1.5,1) node {$\sigma*v$};
      \end{tikzpicture}
    \end{center}

    \caption{Cone~$\sigma * v$ of a simplex~$\sigma$ with respect to~$v$}
    \label{fig:cone}
  \end{figure}
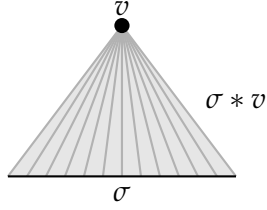

  \begin{defi}[cones of simplices]
    Let $K \subset \R^d$ be a convex subset, let $\sigma \colon
    \Delta^k \longrightarrow K$ be a singular simplex on~$K$ and let
    $v \in K$. Then we define the \emph{cone~$\sigma * v$ of~$\sigma$
      with respect to~$v$} (Figure~\ref{fig:cone}) as the singular simplex
    \begin{align*}
      \sigma * v \colon \Delta^{k+1} & \longrightarrow K \\
      (t_0, \dots, t_{k+1}) & \longmapsto 
      t_{k+1} \cdot v 
      + (1-t_{k+1}) \cdot 
        \sigma \Bigl(\frac1{1-t_{k+1}} \cdot t_0, \dots, 
                     \frac1{1-t_{k+1}} \cdot t_k \Bigr).
    \end{align*}
    This construction extends linearly to all of~$C_k(K;\R)$ and
    passes to~$C_k^\triangle(K;\R)$; these extensions will also be
    denoted by~$\args * v$.
  \end{defi}

  Roughly speaking, we construct triangulations of cubes by first
  triangulating the faces and then coning these simplices with respect
  to the centre of the cube (Figure~\ref{fig:cubestriang}). Using
  naturality, this can be extended to the whole cubical singular chain
  complex.

  We begin by defining $\varphi'_0{}^X$ by
  setting~$\varphi'_0{}^{\square^0}(\id_{\square^0}) := (1 \mapsto 0)$ and
  extending~$\varphi'_0$ to all of~$C_0^\square(\args;\R)$ by
  naturality.

  \label{p:iteratedconing}
  Proceeding inductively, for~$n \in \N$ we assume that
  $\varphi'_n{}^X$ is already constructed for all spaces~$X$ and then set
  \begin{align*}
    \varphi'_{n+1}{}^{X} \colon C_{n+1}^\square(X;\R) &\longrightarrow 
                         C_{n+1}^\triangle(X;\R)
    \\
    (\tau \colon \square^{n+1} \rightarrow X)
    & \longmapsto
    C_{n+1}^\triangle(\tau;\R)
    \biggl( \sum_{j=1}^{n+1} (-1)^{j}\cdot\varphi'_n{}^{\square^n} (\square^{n+1}_{(j,0)} - \square^{n+1}_{(j,1)}) * v_{n+1}
    \biggr);
  \end{align*}
  
  here, $\square^{n+1}_{(j,0)}, \square^{n+1}_{(j,1)} \colon \square^n \longrightarrow
  \square^{n+1}$ denote the face maps of~$\square^{n+1}$ and 
  \[ v_{n+1} := \Bigl(\frac12,\dots,\frac12\Bigr) \in \square^{n+1} \subset \R^{n+1}
  \]
  is the centre of the cube~$\square^{n+1}$. It is not hard to show
  that $\varphi'{}^X$ indeed is a well-defined chain map.

  Because the $n$-cube has exactly~$2 \cdot n$~faces, we inductively obtain 
  \[ \|\varphi'_n{}^X\| \leq 2^n \cdot n!
  \]
  for all spaces~$X$ and all~$n \in \N$.

  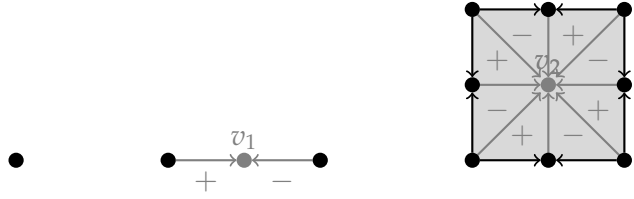
\begin{figure}
    \begin{center}
      \begin{tikzpicture}[thick]
        \fill (0,0) circle (0.1);
        \begin{scope}[shift={(2,0)}]
          \begin{scope}[black!50]
            \draw (1,0) node[anchor=south] {$v_1$};
            \draw (0.5,0) node[anchor=north] {$+$};
            \draw (1.5,0) node[anchor=north] {$-$};
            \draw[->] (0.1,0) -- (0.9,0);
            \draw[->] (1.9,0) -- (1.1,0);
            \fill (1,0) circle (0.1);
          \end{scope}
          \fill (0,0) circle (0.1);
          \fill (2,0) circle (0.1);
        \end{scope}
        \begin{scope}[shift={(6,0)}]
          \begin{scope}[black!15]
            \fill (0,0) rectangle (2,2);
          \end{scope}
          \draw[black!50] (1,1) node[anchor=south] {$v_2$};
          \begin{scope}[black!50]
            \draw[black!50] (1,1) node[anchor=south] {$v_2$};
            \fill (1,1) circle (0.1);
            \draw [->] (0,0) -- (0.9,0.9);
            \draw [->] (1,0) -- (1,0.9);
            \draw [->] (2,0) -- (1.1,0.9);
            \draw [->] (2,1) -- (1.1,1);
            \draw [->] (2,2) -- (1.1,1.1);
            \draw [->] (1,2) -- (1,1.1);
            \draw [->] (0,2) -- (0.9,1.1);
            \draw [->] (0,1) -- (0.9,1);            
            \draw (0.66,0.33) node {$+$};
            \draw (1.33,0.33) node {$-$};
            \draw (1.66,0.66) node {$+$};
            \draw (1.66,1.33) node {$-$};
            \draw (1.33,1.66) node {$+$};
            \draw (0.66,1.66) node {$-$};
            \draw (0.33,1.33) node {$+$};
            \draw (0.33,0.66) node {$-$};
          \end{scope}
          \def\triang{%
            \draw[->] (0.1,0) -- (0.9,0);
            \draw[->] (1.9,0) -- (1.1,0);
            \fill (1,0) circle (0.1);}
          \triang
          \begin{scope}[rotate={90}]
            \triang
          \end{scope}
          \begin{scope}[shift={(2,0)}]
            \begin{scope}[rotate={90}]
              \triang
            \end{scope}
          \end{scope}
          \begin{scope}[shift={(0,2)}]
            \triang
          \end{scope}
          \fill (0,0) circle (0.1);
          \fill (2,0) circle (0.1);
          \fill (0,2) circle (0.1);
          \fill (2,2) circle (0.1);
        \end{scope}
      \end{tikzpicture}
    \end{center}

    \caption{Triangulations of cubes via iterated cones, in dimension~$0,1,2$}
    \label{fig:cubestriang}
  \end{figure}

\begin{cor}\label{cor:svvsqsv}
  For all oriented compact $n$-manifolds~$M$ we have
  \[ \frac1{2^n \cdot n!} \cdot \sv{M,\partial M} 
     \leq \qsv{M,\partial M} \leq \sv{M,\partial M}
  \]
  and for all oriented $n$-manifolds~$M$ without boundary we have
  \[ \frac1{2^n \cdot n!} \cdot \svlf{M^\circ} 
     \leq \qsvlf{M^\circ} \leq \svlf{M^\circ}.
  \]
\end{cor}
\begin{proof}
  By Remark~\ref{rem:eilenbergbound} and the above explicit
  construction (as well as Example~\ref{exa:equivalence}), there
  exists a cubical normed model $(F,\varphi,\psi)$ of~$C_*(\args;\R)$
  that satisfies
  \[ \| \varphi_n^X\| \leq 2^n \cdot n!
     \quad\text{and}\quad 
     \|\psi_n^X\| \leq 1
  \]
  for all spaces~$X$ and all~$n \in \N$. 
  Therefore, Proposition~\ref{prop:equivalence} gives the desired
  bounds.
\end{proof}

A more careful analysis of known triangulations of cubes shows that
the lower bound in dimension~$n \in \N$ can be improved to
$2/({2^{n-1} + n!})$~\cite{waeber}. For simplicity, we will treat
such an improvement only in low dimensions.

\subsection{Cubical models in low dimensions}\label{subsec:subdivision}

We will now discuss improved lower bounds for cubical simplicial
volume in low dimensions.

In the case of surfaces~$M$, it is known that $\qsv M = 1/2 \cdot \sv
M$ holds~\cite{loehplankl}.

\begin{cor}\label{cor:3lowerbound}
  For all oriented compact $3$-manifolds~$M$ and for all oriented
  $3$-manifolds~$N$ without boundary we have
  \[ \qsv{M,\partial M} \geq  \frac15 \cdot \sv{M,\partial M}
     \quad\text{and}\quad
     \qsvlf{N} \geq \frac15 \cdot \svlf{N}.
  \]
\end{cor}
\begin{proof}
  As in the proof of Proposition~\ref{prop:cubex} and
  Corollary~\ref{cor:svvsqsv} it suffices to construct a natural chain
  map~$\varphi \colon C_*^\square(\args;\R) \Longrightarrow
  C_*^\triangle(\args;\R)$ that respects the normalisation and 
  satisfies~$\cdlone{\varphi_3^{\square^3}(\id_{\square^3})} \leq 5$.

  Moreover, in view of the cone construction (or the technique of
  acyclic models) as used in the previous section, it
  suffices to construct the first steps~$\varphi_0, \dots, \varphi_3$
  compatibly with the normalisation -- the higher degrees can then be
  added through iterated coning as on page~\pageref{p:iteratedconing}.

  We first define maps~$T_j \colon Q_j(\args;\R) \Longrightarrow
  C_j(\args;\R)$ for all~$j \in \{0,1,2,3\}$: Here, we use the
  following notation: For an ascending sequence~$n_1, n_2, \dots n_k$
  in~$\N_{>0}$, we write $e_{n_1, \dots, n_k} := \sum_{j=1}^k e_{n_j}
  \in \R^\infty$ for the corresponding sum of standard basis
  vectors. 

  Moreover, for~$v_0,\dots, v_n \in \R^\infty$ we
  write~$[v_0,\dots,v_n] \colon \Delta^n \longrightarrow \R^\infty$
  for the affine singular $n$-simplex with vertices~$v_0, \dots, v_n$
  (in this order).

  \emph{Dimension~$0$:} We set
  \[ T_0^{\square^0}(\id_{\square^0}) := [0] \colon \Delta^0 \longrightarrow \square^0
  \]
  and extend $T_0$ to all of~$Q_0(\args;\R)$ by naturality. Clearly,
  this map satisfies the normalisation condition.

  \emph{Dimension~$1$:} We set 
  \begin{align*}
    T_1^{\square^1}(\id_{\square^1}) 
    := [0,e_1]
    \colon \Delta^1 \longrightarrow \square^1 
  \end{align*}
  and extend~$T_1$ to all of~$Q_1(\args;\R)$ by naturality.
 
  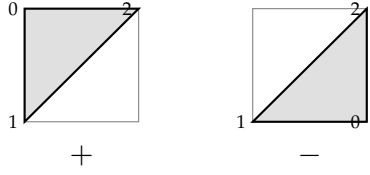
\begin{figure}
    \begin{center}
        \def\cube{%
          \begin{scope}[black!50,thin]
            \draw (0,0) rectangle +(1,1);
          \end{scope}}
        \def\io{(1,0)}
        \def\oo{(0,0)}
        \def\ii{(1,1)}
        \def\oi{(0,1)}
        \def\tetra#1#2#3{%
          \cube
          \fill [black!20,opacity=0.6] #1 -- #2 -- #3 -- cycle; 
          \draw #1 -- #2 -- #3 -- cycle;
          \begin{scope}[shift={(-0.1,0)}]
            \draw #1 node {\tiny$0$};
            \draw #2 node {\tiny$1$};
            \draw #3 node {\tiny$2$};
          \end{scope}%
        }
      \begin{tikzpicture}[x=1.5cm,y=1.5cm,thick]
        \tetra{\oi}{\oo}{\ii}
        \draw (0.5,-0.3) node {$+$};
        \begin{scope}[shift={(2,0)}]
          \tetra{\io}{\oo}{\ii}
        \draw (0.5,-0.3) node {$-$};
        \end{scope}
      \end{tikzpicture}
    \end{center}

    \caption{Triangulating a square with two triangles}
    \label{fig:triangc2}
  \end{figure}

  \emph{Dimension~$2$:} We set (Figure~\ref{fig:triangc2})
  \begin{align*}
    T_2^{\square^2}(\id_{\square^2}) 
    := [e_2, 0, e_{1,2}] - [e_1, 0, e_{1,2}]
    \in C_2(\square^2;\R)
  \end{align*}
  and extend~$T_2$ to all of~$Q_2(\args;\R)$ by naturality.

  \emph{Dimension~$3$:} Using the decomposition of the $3$-cube into
  five tetrahedra depicted in Figure~\ref{fig:triangc3}, we set
  \begin{align*}
    T_3^{\square^3}(\id_{\square^3}) 
    := & \;[0,e_1,e_2,e_3] \\
     - & \;[e_{1,2}, e_1, e_2, e_{1,2,3}] \\
     + & \;[e_{1,3}, e_1, e_3, e_{1,2,3}] \\
     - & \;[e_{2,3},e_2, e_3, e_{1,2,3}] \\
     + & \;[e_e, e_2, e_3, e_{1,2,3}]
    \qquad \in C_3(\square^3;\R),
  \end{align*}
  and extend~$T_3$ to all of~$Q_3(\args;\R)$ by naturality.

  These maps are not yet compatible with the cubical and the ordinary
  singular boundary operators. Therefore, we symmetrise these maps as follows: 
  For~$j \in \N$ let
  \begin{align*}
    \Sigma_j := \frac1{|\Isom(\square^j)|} \cdot 
                \sum_{\pi \in \Isom(\square^j)} (-1)^{\sgn \pi} \cdot C_j(\pi;\R) 
    \colon C_j(\square^j;\R) & \longrightarrow C_j(\square^j;\R), 
  \end{align*}
  where $\sgn \pi \in \{0,1\}$ encodes whether $\pi$ is orientation-preserving 
  or not. Now a straightforward calculation shows that our triangulations fit 
  together in the sense that 
  \[ \partial_j \circ \Sigma_j \circ T_j = \Sigma_{j-1} \circ T_{j-1} \circ \partial^\square_j
  \]
  holds for all~$j \in \{1,2,3\}$. 

  Moreover, it is not hard to see that $\Sigma_j \circ T_j$ induces for 
  all~$j \in \{0,\dots,3\}$ well-defined natural maps
  \[ \varphi_j \colon C_j^\square(\args;\R) \Longrightarrow C_j^\triangle(\args;\R) 
  \]
  that are compatible with the boundary operators. Indeed, the
  well-defined\-ness modulo degenerate cubes can be seen as
  follows: If $\tau \colon \square^j \longrightarrow X$ is a
  degenerate singular cube, then the reflection~$\pi
  \in\Isom(\square^j)$ on the middle hyperplane orthogonal to one of
  the directions of degeneracy of~$\tau$ ensures that every singular
  simplex in~$\Sigma_j \circ T_j(\tau)$ occurs the same number of
  times with positive and negative sign.

  By construction, we have
  \[ \cdlone{\varphi_3^{\square^3} (\id_{\square^3})}
     \leq \frac1{|\Isom(\square^3)|} \cdot |\Isom(\square^3)| \cdot 5 = 5,
  \]
  as desired.
\end{proof}

\begin{figure}
  \begin{center}
    \def\cube{%
      \begin{scope}[black!50,thin]
        \draw (0,0) rectangle +(1,1);
        \draw (0.25,0.5) rectangle +(1,1);
        \draw (0,0) -- (0.25,0.5);
        \draw (1,0) -- +(0.25,0.5);
        \draw (1,1) -- +(0.25,0.5);
        \draw (0,1) -- +(0.25,0.5);
      \end{scope}
    }
    \def\tetra#1#2#3#4{%
      \begin{scope}[black!20,opacity=0.6]
        \fill #1 -- #2 -- #3 -- cycle;
        \fill #2 -- #3 -- #4 -- cycle;
        \fill #1 -- #3 -- #4 -- cycle;
        \fill #1 -- #2 -- #4 -- cycle;
      \end{scope}
      \cube
      \draw #1 -- #2 -- #3 -- #4 -- #1;
      \draw #2 -- #4;
      \draw #1 -- #3;
      \begin{scope}[shift={(-0.1,0)}]
        \draw #1 node {\tiny$0$};
        \draw #2 node {\tiny$1$};
        \draw #3 node {\tiny$2$};
        \draw #4 node {\tiny$3$};
      \end{scope}
    }
    \def\ioo{(1.25,0.5)}
    \def\oio{(0,0)}
    \def\ooi{(0.25,1.5)}
    \def\iio{(1,0)}
    \def\iii{(1,1)}
    \def\oii{(0,1)}
    \def\ooo{(0.25,0.5)}
    \def\ioi{(1.25,1.5)}
    \makebox[0pt]{%
    \begin{tikzpicture}[thick, x=1.5cm,y=1.5cm]
      \tetra{\ooo}{\ioo}{\oio}{\ooi}
      \draw (0.625,-0.3) node {$+$};
      \begin{scope}[shift={(2,0)}]
        \tetra{\iio}{\ioo}{\oio}{\iii}
        \draw (0.625,-0.3) node {$-$};
      \end{scope}
      \begin{scope}[shift={(4,0)}]
        \tetra{\ioi}{\ioo}{\ooi}{\iii}
        \draw (0.625,-0.3) node {$+$};
      \end{scope}
      \begin{scope}[shift={(6,0)}]
        \tetra{\oii}{\oio}{\ooi}{\iii}
        \draw (0.625,-0.3) node {$-$};
      \end{scope}
      \begin{scope}[shift={(8,0)}]
        \tetra{\ioo}{\oio}{\ooi}{\iii}
        \draw (0.625,-0.3) node {$+$};
      \end{scope}
    \end{tikzpicture}}
  \end{center}

  \caption{Triangulating a $3$-cube with five tetrahedra}
  \label{fig:triangc3}
\end{figure}
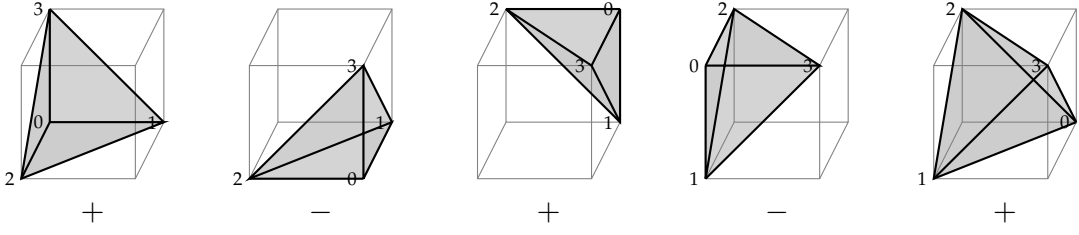

\section{Seifert fibred pieces}\label{sec:seifert}

The goal for this section is to prove that Seifert fibred pieces have
trivial cubical simplicial volume. We will first treat a slightly more
general situation and then obtain the Seifert fibred case as a special
case.

\subsection{Open manifolds with vanishing simplicial volumes}

In the presence of sufficiently nice boundary components, vanishing of
the relative simplicial volume implies the vanishing of the locally
finite simplicial volume of the interior. We first formulate the
boundary condition:

\begin{defi}[UBC~boundary]
  An oriented compact $n$-manifold~$M$ has
  \emph{UBC boundary} if every boundary component
  of~$M$ satisfies~$(n-1)$-UBC (in the sense of Definition~\ref{def:ubc}). 
\end{defi}

For example, all manifolds with amenable fundamental group satisfy UBC
in all degrees~\cite{mm}. This includes, in particular, tori.

\begin{prop}\label{prop:vanishing}
  Let $(F,\varphi,\psi)$ be a normed model of the singular
  chain complex and let $M$ be an oriented compact manifold
  with UBC boundary that satisfies $\sv{M,\partial M} = 0$. Then
  \[ \gsv{M,\partial M}^F = 0 = \sv{M,\partial M}
     \quad\text{and}\quad
     \gsvlf{M^\circ}^F = 0 = \svlf{M^\circ}.
  \]
\end{prop}

\begin{proof}
  In view of the equivalence of vanishing of generalised simplicial
  volumes (Proposition~\ref{prop:equivalence}), it suffices to prove
  that~$\svlf{M^\circ} = 0$.

  Let $n := \dim M$ and $\kappa \in \R_{>0}$ be an $(n-1)$-UBC
  constant for~$\partial M$. Let $\varepsilon \in \R_{>0}$. Because
  of~$\sv{M,\partial M} = 0$ there is a chain~$c \in C_n(M;\R)$ with
  $\partial c \in C_{n-1}(\partial M;\R)$ that represents~$\fclr{M,\partial M}$
  and satisfies
  \[ \clone c \leq \varepsilon. 
  \]
  In particular, $\clone {\partial c} \leq (n+1) \cdot \clone c \leq
  (n+1) \cdot \varepsilon$. 

  Furthermore, we have~$\sv {\partial M} \leq (n+1) \cdot
  \sv{M,\partial M} = 0$ (Proposition~\ref{prop:svboundary}). Using
  the uniform boundary condition on~$\partial M$, we hence find
  fundamental cycles~$(z_k)_{k \in \N} \subset C_{n-1}(\partial M;\R)$
  and chains~$(b_k)_{k \in \N} \subset C_n(\partial M;\R)$ such that
  \begin{align*}
    z_0 & = \partial c, \\
    \fa{k \in \N} \clone{z_k} & \leq \frac1{2^k} \cdot (n+1) \cdot \varepsilon,\\
    \fa{k \in \N} \partial b_k & = z_{k+1} - z_k, \\
    \fa{k \in \N} \clone{b_k} & \leq \kappa \cdot \frac 2{2^k} \cdot (n+1) \cdot \varepsilon.
  \end{align*}
  Stretching out these chains~$(b_k)_{k \in \N}$ along a collar of~$M$
  and combining the resulting chain with~$c$ produces a locally finite
  fundamental cycle of~$M^\circ$ and implies
  that~\cite[Proposition~6.4, proof of Theorem~6.1]{loehphd}\cite[Section~6.2]{loehl1}
  \begin{align*} 
      \svlf{M^\circ}
    & \leq \clone c + \sum_{k \in \N} \clone{b_k} 
          + n \cdot \sum_{k \in \N} \clone{z_{k+1}}
          \\
    & \leq \varepsilon + 4 \cdot \kappa \cdot (n+1) \cdot \varepsilon
           + 2 \cdot n \cdot (n+1) \cdot \varepsilon.
  \end{align*}
  Taking~$\varepsilon\rightarrow 0$ shows that~$\svlf{M^\circ} = 0$, as desired. 
\end{proof}

Whether Proposition~\ref{prop:vanishing} holds without the additional
UBC condition is an open problem~\cite[p.~86]{loehphd}.

\subsection{Vanishing of simplicial volumes of Seifert fibred pieces}

Let us recall the definition of Seifert fibred $3$-manifolds~\cite{AFW}.

\begin{defi}[standard fibred torus]
The \emph{standard fibred torus} corresponding to a pair~$(a,b)$ of coprime integers with 
$a>0$ is the surface bundle of the automorphism of a disk given by rotation by~$2\pi b/a$,
equipped with the natural fibering by circles. 
\end{defi}

\begin{defi}[Seifert fibred manifold]
A \emph{Seifert fibred} $3$-manifold is a compact $3$-manifold $M$ with a decomposition
of $M$ into disjoint circles, called \emph{fibers}, such that each circle has a tubular neighbourhood in $M$ which is isomorphic to a standard fibred torus.
\end{defi}

Notice that in an oriented Seifert fibred manifold the Seifert fibred structure defines a product structure on each boundary component, which implies that each boundary component is an $S^1$-bundle; thus, each boundary component is a torus.

As a consequence of Proposition~\ref{prop:vanishing} and the vanishing of classical 
simplicial volume we obtain:

\begin{cor}[Seifert fibred pieces]\label{cor:seifert}
  Let $(F,\varphi,\psi)$ be a normed model of the singular chain complex 
  and let $M$ be a Seifert fibred $3$-manifold. Then 
  \[ \gsvlf{M^\circ}^F = 0. 
  \]
  In particular, $\qsvlf{M^\circ} = 0$.
\end{cor}
\begin{proof}
  It is well known that $\sv{M,\partial M} = 0$ holds for Seifert
  fibred $3$-mani\-folds~\cite[Corollary~6.5.3]{thurston}. 
  Because the boundary components of~$M$ are tori, $M$ has
  UBC boundary, and thus Proposition~\ref{prop:vanishing} shows
  that~$\gsvlf{M^\circ}^F = 0$.
\end{proof}

 \section{Hyperbolic pieces}\label{sec:hyp}

 As next step, we discuss the hyperbolic case: we have already recalled
 in Theorem~\ref{thm:svhyp} that for a complete oriented hyperbolic
 $n$-manifold of finite volume without boundary the ordinary simplicial
 volume coincides with the ratio between the Riemannian volume and
 $v_n^\triangle$.  The constant~$v_n^\triangle$, depending only on the
 dimension~$n$ of the manifold, is the supremum of volumes of geodesic
 $n$-simplices in the hyperbolic space~$\hyp^n$, which is equal to the
 volume of the ideal and regular geodesic $n$-simplex~\cite{HM,pey}.

In the cubical case, the situation is more involved because the
interaction between combinatorics and the geometry of hyperplanes
in~$\hyp^n$ is more delicate than in the simplicial case. We review
hyperbolic cubes in Section~\ref{subsec:hypcubes}. In
Section~\ref{subsec:hypupper}, we prove an upper bound for cubical
simplicial volume of hyperbolic manifolds (using a cubical smearing),
and in Section~\ref{subsec:hyplower}, we prove a lower bound (using a
cubical straightening). In dimension~$3$, we have matching bounds,
which gives the exact value (Corollary~\ref{cor:hyp3}). In
Section~\ref{subsec:hyppp}, we briefly discuss proportionality for
cubical simplicial volume of hyperbolic manifolds.

For these arguments, we will use the concrete description of locally
finite cubical chains through infinite sums of singular cubes and the
corresponding description of locally finite cubical simplicial volume
through the $\ell^1$-norm of such chains, in analogy with the
simplicial case (Example~\ref{exa:lfalternative}).

\subsection{Hyperbolic cubes}\label{subsec:hypcubes}

We first discuss different types of cubes in~$\hyp^n$. The $2^k$~vertices 
of the standard $k$-cube~$\square^k$ will be denoted by~$(v_j)_{j \in J^k}$, 
where $J^k := \{0,1\}^k$. Taking iterated geodesics produces straight cubes:

\begin{defi}[straight cube]
  Let $X$ be a uniquely geodesic metric space (e.g.,~$\hyp^n$), let $k
  \in \N$ and let $x = (x_j)_{j \in J^k}$ be a family of points
  in~$X$. Then the \emph{straight cube~$[x]^\square$ with vertices~$x$
    in~$X$} is defined inductively via
  \begin{align*}
    [x]^\square \colon \square^k & \longrightarrow X \\
    (t_1,\dots, t_k) & \longmapsto 
    (1-t_k) \cdot \bigl[(x_{j0})_{j \in J^{k-1}}\bigr]^\square(t_1,\dots,t_{k-1}) 
    \\
    & \phantom{\longmapsto}\!\!
    + t_k \cdot \bigl[(x_{j1})_{j \in J^{k-1}}\bigr]^\square(t_1,\dots,t_{k-1});
  \end{align*}
  if $k=0$ we just define the straight cube as the corresponding single point.
\end{defi}

In other words, straight cubes are defined as iterated geodesic joins
of opposite faces. In particular, straight cubes in~$\hyp^n$ are
smooth~\cite[Section~2.1.1]{loehsauer}. 

\subsubsection{Volumes of cubes}

For a smooth cube~$c \colon
\square^n \longrightarrow \hyp^n$ we define the \emph{signed volume} by
\[ \vol^{\text{alg}}_{\hyp^n} c := \int_{\square^n} c^* d \Omega, 
\]
where $\Omega$ is the volume form on $\hyp^n$, and the \emph{volume} by $\vol_{\hyp^n}c:=|\vol^{\text{alg}}_{\hyp^n}c|$.

A straight cube is \emph{(non)-degenerate} if its volume is (non)-zero. If all vertices of a non-degenerate cube are different,  the cube is \emph{strongly non-de\-generate}.
If the signed volume is positive (respectively negative) the cube is \emph{positively} (respectively~\emph{negatively}) \emph{oriented}.

\begin{defi}[volume function]
Identifying a straight cube in $\hyp^n$ with its vertices we define the following \emph{signed volume function}
$$
\begin{array}{rrcl}
\vol^{\text{alg}}:&({\hyp}^n)^{2^n}&\longrightarrow &\R\\
& x=(x_j)_{j\in J^n}&\longmapsto & \vol^{\text{alg}}_{\hyp^n}([x]^\square).
\end{array}
$$
\end{defi}
Similarly, we define the volume function $\vol:=|\vol^{\text{alg}}|$. 

\begin{rem}\label{remark:cont}
Notice that the signed volume function is continuous with respect to the topology on 
$({\hyp}^n)^{2^n}$ induced by the distance~$\times d_{\hyp^n}$. Indeed, hyperbolic geodesics continuously depend on their endpoints.
\end{rem}

By definition of straight cubes, we can express the diameter of $[x]^\square$  in terms
of the distances between the elements of $x$: 
\begin{equation}\label{diam}
\diam_{\hyp^n} [x]^\square=\max_{i, j\in J^k} d_{\hyp^n}(x_i, x_j).
\end{equation}

One major difficulty with cubes is that not all straight cubes are
geodesic in the following sense:

\begin{defi}[geodesic cube]
  A straight cube in~$\hyp^n$ is \emph{geodesic} if each $(n-1)$-face lies in
  a hyperbolic hyperplane of~$\hyp^n$. 
\end{defi}

\begin{prop}[volume bounds for cubes]\label{volume}
  Let $n\in \N$. Then the numbers
  \begin{align*}
    v_n^\square & := 
    \sup \bigl\{ \vol_{\hyp^n} c \bigm| \text{$c$ is a geodesic $n$-cube in~$\hyp^n$}\bigr\} 
    \\
    w_n^\square & := 
    \sup \bigl\{ \vol_{\hyp^n} c \bigm| \text{$c$ is a straight $n$-cube in~$\hyp^n$}\bigr\}     
  \end{align*}
  are finite. 
\end{prop}
\begin{proof}
  Let $T(n)$ be the number of $n$-simplices needed to triangulate 
  convex hulls of $2^n$~points in~$\hyp^n$. For a straight cube~$c:\square^n\rightarrow \hyp^n$, we denote by $\text{Conv}(c(v_j)_{j\in J^n})$
the convex hull of the vertices $(c(v_j))_{j\in J^n}$.
By definition of straight cubes, we have $\im(c)\subseteq \text{Conv}(c(v_j)_{j\in J^n})$, and therefore it follows that
$v_n^\square\leq w_n^\square\leq T(n)\cdot v_n^\triangle<\infty$.
\end{proof}

By definition, clearly~$v_n^\square \leq w_n^\square$. However, the
exact relation in high dimensions seems to be unknown.

\subsubsection{Ideal geodesic cubes}
Let $\overline{\hyp^n}=\hyp^n\cup\partial \hyp^n$ be the standard compactification of the hyperbolic space.
We extend the notion of straight cubes to ideal straight cubes.
If $x= (x_j)_{j \in J^k}$ is a family of points in $\partial{\hyp^n}$, then the straight cube $[x]^\square$ is \emph{ideal}. 

Computing the supremum of the volume of geodesic cubes we may include also geodesic cubes with vertices on the boundary by setting
$$
\overline{v}_n^\square:=\sup \bigl\{ \vol_{\hyp^n} c \bigm| \text{$c$ is a geodesic $n$-cube in~$\overline{\hyp^n}$}\bigr\}. $$ 

\begin{lem}\label{lem:ideal}
For every $n \in \N_{>0}$, we have that $v_n^\square=\overline{v}_n^\square$.
\end{lem}
\begin{proof}
Using the same construction of Remark \ref{max} we deduce that, given a geodesic cube having some vertices on $\bb \hyp^n$, there exists an ideal one with the same vertices on $\bb \hyp^n$ containing it. Therefore,
it suffices to show that for every positively oriented ideal geodesic $n$-cube $c$ and all $\varepsilon >0$
there exists a geodesic $n$-cube $c_\varepsilon:\square^n\rightarrow \hyp^n$ with $\vol^{\rm alg}_{\hyp^n} c_\varepsilon \geq\vol^{\rm alg}_{\hyp^n} c - \varepsilon$.

First of all notice that if we pick an ideal cube that is not strongly
non-degenerate there exists a strongly non-degenerate one of bigger
volume.  Therefore, we can restrict ourselves to considering strongly
non-degenerate ideal cubes.
 
Let $B^n\subset \R^n$  be the Poincar\'e disk model with origin $O$ and boundary~$S^{n-1}$.
Let $c$ be a positively oriented strongly non-degenerate ideal geodesic $n$-cube and let $(c(v_j))_{j\in J^n}$ be its vertices
in $S^{n-1}$. Let $\alpha_j:[0,\infty]\rightarrow \mathbb{H}^n$ be the geodesic ray (with constant speed parametrization) 
that starts at the origin~$O$ and is asymptotic to $c(v_j)$. For~$L \in \R_{>0}$ we then define the straight cube
\[ c_L := \bigl[ (\alpha_j(L))_{j \in J^n} \bigr]^\square 
\]
in~$\hyp^n$. It is easy to verify that~$\lim_{L \rightarrow \infty}
\vol^{\rm alg}_{\hyp^n} c_L = \vol^{\rm alg}_{\hyp^n} c$. Therefore, for
every~$\varepsilon > 0$ there exists~$L \in \R_{>0}$ with~$\vol^{\rm
  alg}_{\hyp^n} c_{L}\geq \vol^{\rm alg}_{\hyp^n}c-\varepsilon$; in
particular, $c_L$ is positively oriented as soon as~$\epsilon <
\vol^{\rm alg}_{\hyp^n} c$.
\end{proof}

\begin{question}
Is there a geodesic $n$-cube realizing the maximum of the volume?
\end{question}

\begin {rem}\label{max}
One may notice that the maximum of the volume of geodesic $n$-cubes (possibly with some vertices on $\bb \hyp^n$), if any, is attained by ideal $n$-cubes. 
Indeed, given a geodesic $n$-cube $c$, let $v$ be a point in the internal part of $\im(c)$. For each vertex~$c(v_j)$ of $c$ consider the geodesic ray starting from $v$ and passing through~$c(v_j)$, and pick its endpoint $c_\infty^j$.
The ideal $n$-cube $[(c_\infty^j)_j]^\square$ contains~$c$. Then, for every geodesic $n$-cube  there exists an ideal one with bigger volume.
\end{rem}

\begin{question}
If there exists an ideal $n$-cube realizing the maximum, is it a regular $n$-cube?
\end{question}

\begin{exa}\label{exa:coxeter}
In dimension $3$, Coxeter~\cite{coxeter} observed that the ideal regular $3$-cube can
be triangulated by five ideal regular tetrahedra. Then the ideal
regular $3$-cube realizes the supremum of the volumes of the
geodesic $3$-cubes and $v_3^\square=5\cdot v_3^\triangle$.
\end{exa}

\begin{exa}
In dimension $4$, we have that $15\cdot v_4^\triangle\leq \overline{v}_4^\square \leq 16 \cdot v_4^\triangle$ \cite{smi,ma}.
\end{exa}

\subsubsection{Approximation of cubes}
We will need that we can approximate large geodesic cubes
well enough:

\begin{prop}[approximation of cubes]\label{prop:approxcubes}
  Let $n \in \N$ and let $\varepsilon \in (0,v_n^\square/2)$. Then
  there exists a positively oriented geodesic cube~$c_\varepsilon$
  in~$\hyp^n$ and $\delta >0$ with the following properties:
  \begin{enumerate}
    \item\label{it:prop1} 
      The cube~$c_\varepsilon$ has large volume, i.e., $\vol_{\hyp^n} c_\varepsilon
      \geq v_n^\square - \varepsilon$.
    \item \label{it:prop2} 
      Every straight cube close to~$c_\varepsilon$ has large volume and 
      is positively oriented, i.e., for all straight cubes~$c
      \colon \square^n \longrightarrow \hyp^n$ with 
      \[ \max_{j \in J^n} d_{\hyp^n} \bigl( c(v_j), c_\varepsilon(v_j) \bigr) \leq \delta
      \]
      we have $\vol_{\hyp^n} c \geq
      \vol_{\hyp^n} c_\varepsilon - \varepsilon$ and $c$ is positively oriented.
  \end{enumerate}
\end{prop}
\begin{proof}
By definition of $v_n^\square$, there is a geodesic $n$-cube $c_\varepsilon$ in~$\hyp^n$ such that $\vol_{\hyp^n}c_\varepsilon \geq v_n^\square-\varepsilon$. Without loss of generality we may suppose that $c_\varepsilon$ is positively oriented.

Let $\square_n(\hyp^n)$ be the set of straight $n$-cubes in $\hyp^n$. The signed volume function is continuous on~$\square_n(\hyp^n)$ (Remark~\ref{remark:cont}), and therefore there exists a~$\delta >0$ such that 
$$
\fa{c \in \square_n(\hyp^n)} \times d_{\hyp^n}(c, c_\varepsilon)< \delta 
\Longrightarrow |\vol^{\rm alg}_{\hyp^n}c_\varepsilon -\vol^{\rm alg}_{\hyp^n}c|<\varepsilon.
$$
Notice that $\times d_{\hyp^n}(c,c_\varepsilon)< \delta$ is equivalent to  $\max_{j \in J^n} d_{\hyp^n}(c(v_j),c_\varepsilon(v_j))<\delta$; finally, $
|\vol^{\text{alg}}_{\hyp^n} c-\vol^{\text{alg}}_{\hyp^n} c_\varepsilon|< \varepsilon$ implies that $$\vol^{\text{alg}}_{\hyp^n} c\,>\,\vol^{\text{alg}}_{\hyp^n} c_\varepsilon-\varepsilon\,\geq\,v_n^\square-2\cdot\varepsilon\,>\,0$$
and hence any such~$c$ is positively oriented.
\end{proof}

\subsection{Hyperbolic pieces -- upper bound}\label{subsec:hypupper}

We will now establish an upper bound for cubical simplicial volume of hyperbolic manifolds in terms of geodesic cubes: 

\begin{thm}\label{thm:hypupper}
  Let $M$ be a complete oriented hyperbolic $n$-manifold of finite volume. 
  Then
  \[ \qsvlf M \leq \frac{\vol (M)}{v_n^\square}. 
  \]
\end{thm}

Before getting into the details of the proof, we mention that this 
allows to calculate cubical simplicial volume of hyperbolic $3$-manifolds:

\begin{cor}\label{cor:hyp3}
  Let $M$ be a complete oriented hyperbolic $3$-manifold of finite volume 
  without boundary. Then
  \[ \qsvlf M = \frac 15 \cdot \svlf M. 
  \]
\end{cor}
\begin{proof}As we have already mentioned
  it is a well-known fact from hyperbolic geometry that $v_3^\square =
  5 \cdot v_3^\triangle$ holds. Hence, Theorem \ref{thm:svhyp}, Lemma \ref{lem:ideal},
  Theorem \ref{thm:hypupper}, and Corollary~\ref{cor:3lowerbound}
  yield the conclusion.
\end{proof}

We will now prove Theorem~\ref{thm:hypupper}, using a discrete cubical
version of the smearing construction by Thurston that already proved
useful for ordinary simplicial volume of hyperbolic
manifolds~\cite{thurston,FP}. 

Let us first fix some notation: Let $M$ be a complete oriented
hyperbolic $n$-manifold of finite volume without boundary. Because
volume and cubical simplicial volume are additive with respect to
connected components, we may assume without loss of generality that
$M$ is connected. Let $\Gamma$ be the fundamental group of~$M$, and
let $\pi \colon \hyp^n \longrightarrow M$ be the universal covering
map. Then $\Gamma$ acts isometrically and properly discontinuously via
deck transformations on~$\hyp^n$. 

Let $G$ be the isometry group of~$\hyp^n$, endowed with the Haar
measure~$\mu_G$; we normalise the Haar measure~$\mu_G$ according to
the volume~$\vol_{\hyp^n}$ on~$\hyp^n$ \cite[Lemma 11.6.4]{ratcliffe};
the group $G$ is
unimodular~\cite[Proposition~C.4.11]{benedettipetronio}, whence
$\mu_G$ is bi-invariant.  Moreover, we let $G^+$ and $G^-$ be the
subset of orientation preserving (or reversing, respectively)
isometries. 

Roughly speaking, the smearing works as follows: For small~$\varepsilon$ we pick
a geodesic $n$-cube~$c_\varepsilon$ as provided by Proposition~\ref{prop:approxcubes},  
and we would want to consider the chain~$\sum_{[g] \in
  \Gamma \setminus G} \pi \circ (g\cdot c_\varepsilon)$. Because this sum is not 
locally finite, we discretise the process, using so-called~$\Gamma$-nets. 

Let $R \in \R_{>0}$. A \emph{$\Gamma$-net of mesh size
  at most~$R$ in~$\hyp^n$} is given by a discrete
subset~$\Lambda\subset \hyp^n$, called set of
\emph{vertices}, and a collection of Borel sets
$(B_x)_{x\in \Lambda}$ in~$\hyp^n$, called \emph{cells}, such that
the following conditions hold:
\begin{enumerate}
  \item\label{zero} The quotient~$\pi(\Lambda) \subset M$ is locally 
    finite. 
  \item\label{one} For all~$x \in \Lambda$ we have~$x \in B_x$.
    Moreover, the sets~$(B_x)_{x \in \Lambda}$ are pairwise disjoint 
    and
    $\hyp^n=\bigcup_{x\in\Lambda} {B}_x$.  
  \item\label{two} 
    For all~$\gamma \in \Gamma$, $x \in \Lambda$ we have
    $\gamma \cdot x \in \Lambda$ 
    and 
    $\gamma \cdot B_x = {B}_{\gamma \cdot x }$. 
  \item\label{three}
    For all~$x \in \Lambda$ we have~$\diam_{\hyp^n} B_x \leq R$.
\end{enumerate}

\begin{lem}
  For every~$R \in \R_{>0}$ and every finite set~$S \subset \hyp^n$, 
  there exists a $\Gamma$-net of mesh size at most~$R$ in~$\hyp^n$ such 
  that the points of~$S$ are contained in the interior of the cells.
\end{lem}
\begin{proof}
  Such a net can be obtained by lifting to~$\hyp^n$ and translating
  with~$\Gamma$ a decomposition of $M$ into small Borel
  sets~\cite[Lemma~3.10]{FP}.
\end{proof}

We now let $\varepsilon \in (0, 1/2\cdot v_n^\square)$. By
Proposition~\ref{prop:approxcubes}, there exists $\delta>0$ and a geodesic
cube~$c_\varepsilon \colon \square^n \longrightarrow \hyp^n$ of large
volume with good approximation properties. By the previous lemma,
there is a $\Gamma$-net~$(\Lambda, (B_x)_{x \in \Lambda})$ of mesh
size at most~$\delta$ in~$\hyp^n$ such that the vertices
of~$c_\varepsilon$ are in the interior of their cells.  Finally, we
can perform the actual smearing construction -- smearing the model
cube~$c_\varepsilon$ over all of~$\hyp^n$ and whence~$M$: For $x \in
\Lambda^{J^n}$ we consider the associated straight cube
\[ \sigma_x := [x]^\square \colon \square^n \longrightarrow \hyp^n 
\]
and the Borel set 
\begin{align*}
  \Omega_{\varepsilon,x}^\pm := 
  \bigl\{ g\in G^\pm 
  \bigm|  \fa{j \in J^n}\  g \cdot c_\varepsilon(v_j) \in B_{x_j}
  \bigr\} 
  \subset G,
\end{align*}
and we abbreviate~$a_{\varepsilon,x}^{\pm} := \mu_G(\Omega_{\varepsilon,x}^\pm) \in \R_{\geq 0}$. 
The group~$\Gamma$ acts diagonally on~$\widetilde X := \Lambda^{J^n}$, and we write~$X$ 
for the corresponding quotient space. 

\begin{lem}[smearing]\label{lem:smearing}
  In this situation, the \emph{smeared chain} 
  \[ z_\varepsilon := \sum_{[x] \in X} ( a^+_{\varepsilon,x} \cdot \pi \circ \sigma_x
                       - a^-_{\varepsilon,x} \cdot \pi \circ \sigma_x) 
  \]
  has the following properties:
  \begin{enumerate}
  \item\label{it:lem1} 
    The definition of~$z_\varepsilon$ does not depend on the choice of
    representatives~$x$ of the orbits~$[x] \in X$.
  \item\label{it:lem2} 
    The chain~$z_\varepsilon$ indeed is a locally finite chain on~$M$.
  \item\label{it:lem3} 
    The chain~$z_\varepsilon$ is a cycle.
  \item\label{it:lem4} 
    The locally finite cubical cycle~$z_\varepsilon$ represents a non-zero 
    multiple of the cubical (locally finite) fundamental class of~$M$. 
  \end{enumerate}
\end{lem}
\begin{proof}
  \emph{Ad~\ref{it:lem1}.}
  By construction, we have~$\sigma_{\gamma \cdot x} = \gamma \cdot \sigma_x$ 
  for all~$x \in \widetilde X$,~$\gamma \in \Gamma$, and hence
  $\pi \circ \sigma_{\gamma \cdot x} = \pi \circ \sigma_x. 
  $
  Furthermore, $a_{\varepsilon,\gamma \cdot x}^\pm = a_{\varepsilon,x}^\pm$ by the 
  equivariance of the net and invariance of the Haar measure~$\mu_G$.
  
  \emph{Ad~\ref{it:lem2}.} 
  This is a size argument: We set 
  $Y := \{ x \in \widetilde X
          \mid a_{\varepsilon,x}^+ \neq 0 \text{ or } a_{\varepsilon,x}^- \neq 0\}. 
  $ 
  If $x \in Y$, then there is a~$g \in G$ with
  \[ \fa{j \in J^n} g \cdot c_{\varepsilon}(v_j) \in B_{x_j} 
  \]
  and so~$\diam_{\hyp^n} x \leq 2 \cdot (2 \cdot \delta + \diam_{\hyp^n} \im c_\varepsilon)$. 
  Moreover, the diameter of straight cubes in~$\hyp^n$ is controlled in terms of the 
  diameter of its set of vertices (equation~\eqref{diam}). Hence, there is~$L \in \R_{>0}$ such that
  \[ \fa{x \in Y} \diam_{\hyp^n} \im \sigma_x \leq L. 
  \]
  So, if $K \subset M$ is compact, then 
  \[ Y_K := 
     \bigl\{ x \in Y
     \bigm| \im (\pi \circ \sigma_x) \cap K \neq \emptyset
     \bigr\}
     \subset 
     \bigl(\Lambda \cap \pi^{-1}(B_{2 \cdot L}(K))\bigr)^{J^n}.
  \]
  Now the facts that $\pi(\Lambda)$ is locally finite, that the
  $\Gamma$-action on~$\hyp^n$ is properly discontinuous, and that
  tuples in~$Y$ have diameter at most~$L$ show that the
  quotient~$\Gamma \setminus Y_K$ is finite. 
  Therefore, the infinite chain~$z_\varepsilon$ is locally finite.

  \emph{Ad~\ref{it:lem3}.} 
  We use Thurston's reflection trick to show that $z_\varepsilon$ is a
  cycle: At this point it is crucial that every face of the model
  cube~$c_\varepsilon$ lies in a hyperplane. More precisely, for
  all~$k \in \{1,\dots,n\}$ and~$i \in \{0,1\}$ let $\rho_{k,i} \in G$
  be the hyperbolic reflection at a hyperplane that contains the
  $(k,i)$-face of~$c_\varepsilon$.

  We can now argue similarly to the simplicial
  case~\cite[p.~116]{benedettipetronio}: Clearly, in the expanded
  expression~$\partial^\square_n z_\varepsilon$, only $(n-1)$-cubes of
  the form~$\pi \circ \sigma_y$ with~$y \in \Lambda^{J^{n-1}}$
  occur. One easily deduces from the construction of~$z_\varepsilon$
  that such a cube has the coefficient
  \begin{align*}
    b_{\varepsilon,y} := 
    \sum_{k=1}^n \sum_{i=0}^1 (-1)^{k+i} 
    \cdot \Bigl(& \;\mu_G \bigl\{ g \in G^+ 
                \bigm| \fa{j \in J^{n-1}} g \cdot c_\varepsilon(v_{j +_k i}) \in B_{y_j}
                \bigr\}
                \\
        - & \;\mu_G \bigl\{ g \in G^- 
                \bigm| \fa{j \in J^{n-1}} g \cdot c_\varepsilon(v_{j +_k i}) \in B_{y_j}
                \bigr\}\Bigr)
  \end{align*}
  in~$\partial^\square_n z_\varepsilon$; here, $j +_k i$ denotes the $n$-tuple that 
  results if $i$ is inserted at position~$k$ into~$j$, and we used $y$ as the 
  representative of its own $\Gamma$-orbit.

  Let $k \in \{1,\dots,n\}$ and $i \in \{0,1\}$. 
  Using $G^- = G^+ \cdot \rho_{k,i}$, the bi-invariance of~$\mu_G$, 
  and the fact that $\rho_{k,i}$ fixes the $(k,i)$-face of~$c_\varepsilon$, 
  we obtain 
  \begin{align*}
      & \ \mu_G \bigl\{ g \in G^- 
                \bigm| \fa{j \in J^{n-1}} g \cdot c_\varepsilon(v_{j +_k i}) \in B_{y_j}
                \bigr\}
   \\ = 
      & \ \mu_G \bigl\{ g \in G^+ 
                \bigm| \fa{j \in J^{n-1}} g \cdot \rho_{k,i} \cdot c_\varepsilon(v_{j +_k i}) \in B_{y_j}
                \bigr\}
   \\ = 
      & \ \mu_G \bigl\{ g \in G^+ 
                \bigm| \fa{j \in J^{n-1}} g \cdot c_\varepsilon(v_{j +_k i}) \in B_{y_j}
                \bigr\}
  \end{align*}
  Hence, $b_{\varepsilon,y} = 0$. Therefore, $z_\varepsilon$ is a cycle.

  \emph{Ad~\ref{it:lem4}.}  It suffices to check the claim locally. To
  this end, let $m\in M$ be chosen in such a way that it $\pi$-lifts
  to the interior of~$\im c_\varepsilon$. We now show that
  $z_\varepsilon$ represents a non-trivial class in~$H_n^\square(M,M
  \setminus \{m\};\R)$: We choose~$\xi \in X$ so that $\id \in
  \Omega^+_{\varepsilon,\xi}$. Then $z_\varepsilon$ represents 
  in~$H_n^\square(M,M\setminus \{m\};\R)$ the class given by 
  the (finite) sum
  \[ z_{\varepsilon,m} := \sum_{\substack{[x] \in X,\\ m \in \im \pi \circ \sigma_x}}
     (a_{\varepsilon,x}^+ \cdot \pi\circ \sigma_x 
     - a_{\varepsilon,x}^- \cdot \pi\circ \sigma_x).
  \] 
  By choice of~$m$, the cube~$\pi \circ \sigma_{\xi}$ is a relative
  cycle for~$(M,M\setminus \{m\})$ and represents the orientation
  generator of~$H_n^\square(M,M\setminus \{m\};\R)$. Moreover, because
  the vertices of the cube~$c_\varepsilon$ lie in the interior of
  their cells in the chosen $\Gamma$-net, it is not hard to see that
  $a^+_{\varepsilon,\xi} > 0$. 

  On the other hand, for all~$x \in X$,
  by construction, $\pi\circ \sigma_x$ is positively oriented
  if~$a^+_{\varepsilon,x} \neq 0$ and negatively oriented
  if~$a^-_{\varepsilon,x} \neq 0$. Hence, the remaining terms in the
  defining sum for~$z_{\varepsilon,m}$ represent a non-negative multiple 
  of the orientation generator of~$H_n(M,M\setminus \{m\};\R)$. 
  Therefore, $z_{\varepsilon,m}$ does not represent the trivial class, and 
  so also~$z_\varepsilon$ does not represent the trivial class.
\end{proof}

Using Lemma~\ref{lem:smearing}, it is easy to complete the proof of
Theorem~\ref{thm:hypupper}: 

\begin{proof}[Proof of Theorem~\ref{thm:hypupper}]
  Let $\varepsilon \in (0,1/2\cdot v_n^\square)$. The smeared locally 
  finite cycle~$z_\varepsilon$ of~$M$ constructed in Lemma~\ref{lem:smearing} 
  has finite $\ell^1$-norm because
  \[  |z_\varepsilon|^\square_1 \leq \sum_{[x] \in X} (a^+_{\varepsilon,x} + a^-_{\varepsilon,x}) 
       \leq  2 \cdot \vol (M).
  \]
  By Lemma~\ref{lem:smearing}, $z_\varepsilon$ represents a non-zero 
  multiple of the locally finite cubical fundamental class of~$M$; 
  let $\alpha \in \R \setminus \{0\}$ be this multiple. 
  Moreover, the straight cubes in~$z_\varepsilon$ that occur with non-zero
  coefficients have uniformly bounded diameter (as was shown in the
  proof of Lemma~\ref{lem:smearing}) and so can be seen to be
  uniformly Lipschitz~\cite[Section~2.1.1]{loehsauer}. 
  Similarly 
  to the simplicial case, we then have
  \[ \int_{z_\varepsilon} d\Omega_M = \alpha \cdot \vol (M), 
  \]
where $\Omega_M$ is the volume form on $M$.
  Therefore, using the orientation behaviour of the straight
  cubes~$\sigma_x$ (as in the proof of Lemma~\ref{lem:smearing}) and
  the approximation properties of the model cube~$c_\varepsilon$
  underlying the construction of~$z_\varepsilon$, we obtain
  \begin{align*}
    \alpha \cdot \vol(M) 
    & = \sum_{[x] \in X} 
        \biggl( a_{\varepsilon,x}^+ \cdot \int \sigma_x^* d\Omega
              - a_{\varepsilon,x}^- \cdot \int \sigma_x^* d\Omega 
        \biggr)
        \\
    & = \sum_{[x] \in X}
        \bigl( a^+_{\varepsilon,x} \cdot \vol_{\hyp^n} \sigma_x 
             + a^-_{\varepsilon,x} \cdot \vol_{\hyp^n} \sigma_x
        \bigr)
        \\
    & \geq \sum_{[x] \in X} (a_{\varepsilon,x}^+ + a_{\varepsilon,x}^-) \cdot (v_n^\square - 2\cdot \varepsilon)
        \\
    & = |z_\varepsilon|_1^\square \cdot (v_n^\square - 2 \cdot \varepsilon).
  \end{align*}
  This implies
  \begin{align*}
    \qsvlf M \leq \frac1\alpha \cdot |z_\varepsilon|^\square_1
             \leq \frac{\vol (M)}{v_n^\square - 2 \cdot \varepsilon}.
  \end{align*}
  Taking~$\varepsilon \rightarrow 0$ gives the desired estimate.
\end{proof}

\subsection{Hyperbolic pieces -- lower bound}\label{subsec:hyplower}

Conversely, straight cubes give a lower bound:

\begin{thm}\label{thm:hyplower}
  Let $M$ be a complete oriented hyperbolic $n$-manifold of finite volume. 
  Then
  \[ \qsvlf M \geq \frac{\vol (M)}{w_n^\square}. 
  \]
\end{thm}

We will prove this theorem as in the simplicial case via straightening. 
To every singular $k$-cube $c:\square^k\rightarrow \hyp^n$ in $\hyp^n$ we may associate a straight $k$-cube $[c]^\square:=[(c(v_j))_{j\in J^k}]^\square$
where, as before, $(v_j)_{j\in J^k}$ denotes the vertices of the standard $k$-cube and $J^k:=\{0,1\}^k$.
Hence, by linearity we define the straightening map
$$
[\args]^\square_*:C_*^{\square}(\hyp^n;\R)\longrightarrow C_*^{\square}(\hyp^n;\R),
$$
which is easily seen to be a well-defined chain map. 

Let us now extend the straightening operation to locally finite cubical chains: 
If $M$ is a complete oriented hyperbolic $n$-manifold of finite volume and $c \colon \square^k \longrightarrow M$ is a singular $k$-cube, then we define
\[ [c]^\square_k := \pi \circ [\widetilde c]^\square_k \colon \square^n \longrightarrow M, 
\]
where $\pi \colon \hyp^n \longrightarrow M$ is the universal covering
map and where $\widetilde c$ denotes some $\pi$-lift of~$c$. This construction 
extends to a well-defined chain map
\[ 
[\args]^\square_*:C_*^{\square,\mathrm{lf}}(M;\R)\longrightarrow C_*^{\square,\mathrm{lf}}(M;\R).
\]
In order to see that locally finite chains indeed are mapped to
locally finite chains, one can proceed as follows: We consider a
compact core $N$ of the complete hyperbolic manifold~$M$, i.e., a
subset of~$M$ whose complement $M\backslash N$ in $M$ is a disjoint
union of finitely many geodesically convex cusps of $M$.  Using such a
decomposition it is easy to show that the straightening map indeed
maps locally finite chains to locally finite chains (see for
instance~\cite[Lemma~4.3]{KimKim} for the simplicial case).

This straightening map is chain homotopic to the identity map. Indeed, for 
a singular cube~$c:\square^k\rightarrow \hyp^n$, we consider the straight 
homotopy
$$
\begin{array}{rrcl}
F^k_c:&\square^k\times [0,1]& \longrightarrow & \hyp^n\\
&((t_1,\dots,t_k),s) &\longmapsto & (1-s)\cdot c(t_1,\dots,t_k)+ s \cdot[c]^\square(t_1,\dots,t_k).
\end{array}
$$
We then define a chain homotopy via  
$$
\begin{array}{rrcl}
h_k:& C_k^{\square}(\hyp^n;\R)&\longrightarrow & C_{k+1}^{\square}(\hyp^n;\R)\\
&c &\longmapsto & F_c^k.
\end{array}
$$
It is easy to verify that~$h_*$ is a homotopy between the identity and the straightening map on~$\hyp^n$. Moreover, this argument also descends to the locally 
finite straightening on~$M$. 

\begin{proof}[Proof of Theorem~\ref{thm:hyplower}]
Let $c=\sum_i a_i \cdot c_i$ be a locally finite representative of the fundamental class of $M$ and $\widetilde{c}=\sum_i a_i\cdot \widetilde{c}_i$ be
a lift to the universal cover. Then, $\widetilde c$ is smooth and we have
$$
\vol(M)
=\int_{[\widetilde{c}]_n^\square} d\Omega
= \sum_i a_i \cdot \int_{\square^n} {[\widetilde c_i]^\square_n}^* d\Omega 
\leq\sum_i |a_i|\vol_{\hyp^n} [\widetilde{c}_i]_n^\square\leq |c|_1^\square\cdot w_n^\square,
$$
where for the last inequality we are using Proposition~\ref{volume}.
Passing to the infimum over all the representatives we have
\[\qsvlf M\geq \frac{\vol(M)}{w_n^\square}.
\qedhere
\]
\end{proof}

Unfortunately, in dimension~$n \geq 4$, it is unknown whether the
bounds~$w_n^\square$ and~$v_n^\square$ match and what the exact
relation with the volume of ideal regular cubes is. 

\subsection{Proportionality principle for hyperbolic manifolds}\label{subsec:hyppp}

Analogously, to the case of ordinary simplicial volume, also cubical
simplicial volume of hyperbolic manifolds satisfies a proportionality
principle. For simplicity, we will restrict ourselves to the closed
hyperbolic case. Let $n \in \N$ and let $M$ and $N$ be oriented closed 
connected hyperbolic $n$-manifolds. Then
\[ \frac{\qsv M}{\vol (M)}
   = \frac{\qsv N}{\vol (N)}. 
\]
We sketch how this proportionality can be obtained by a
discrete smearing map (similar to ordinary simplicial
volume~\cite{thurston,loehsauer}):

Let us first fix some notation (similar to Section~\ref{subsec:hypupper}):
Let $\Gamma := \pi_1(M)$, let $G := \Isom^+(\hyp^n)$ and let
$\mu_G$ be the Haar measure on~$G$ normalised by~$\mu_G(G/\Gamma) =
1$. Moreover, let $D \subset \hyp^n$ be a measurable, strict
fundamental domain for the deck transformation action of~$\Gamma$
on~$\hyp^n \cong \widetilde M$ and let $\pi_M \colon \hyp^n
\longrightarrow M$, $\pi_N \colon \hyp^n \longrightarrow N$ be the
universal covering maps. Furthermore, let $\delta \in \R_{>0}$. We
then choose a $\Gamma$-net~$(\Lambda, (B_{x})_{x \in \Lambda})$ of
mesh size at most~$\delta$, and we let $X$ be the quotient
of~$\Lambda^{J^n}$ with respect to the diagonal
$\Gamma$-action. For~$x \in \Lambda^{J^n}$ and a singular
cube~$c \colon \square^n \longrightarrow \hyp^n$, we set
\begin{align*}
  A_{x}(c) & := 
  \bigl\{ g \in G \bigm| \fa{j \in J^n} g \cdot c(v_j) \in B_{x_j} \bigr\} \\
  a_{x}(c) & := \mu_G(A_{x})\\
  c_{x} & := [x]^\square \colon \square^n \longrightarrow \hyp^n.
\end{align*}
Finally we define the discretised smearing map
\begin{align*}
  \smear^\delta_{N,M} \colon C_n^\square(N;\R)
  & \longrightarrow C_n^\square(M;\R)
  \\
  \map(\square^n,N) \ni c & \longmapsto
  \sum_{[x] \in X} a_{x} (\widetilde c) \cdot \pi_M \circ c_{x},
\end{align*}
where $\widetilde c$ is a $\pi_N$-lift
of~$c$. Straightforward calculations show that
$\smear^\delta_{N,M}$ has the following properties:
\begin{itemize}
  \item The map~$\smear_{N,M}^\delta$ is well-defined and induces a well-defined 
    map on cubical singular homology.
  \item We have $\| \smear_{N,M}^\delta\| \leq 1$ with respect
    to the cubical $\ell^1$-norms.
  \item If $c \in C_n^\square(N;\R)$ is a smooth cubical fundamental
    cycle of~$N$, then continuity of the volume of straight cubes
    in~$\hyp^n$ shows that
    \[ \lim_{\delta \rightarrow 0} \int_{\smear_{N,M}^\delta(c)} d\Omega_M 
       = \int_{c} d\Omega_N = \vol N.
    \]
\end{itemize}
An inductive smoothing procedure (e.g., through straightening) shows
that cubical simplicial volume can be computed via smooth cubical
fundamental cycles. Because integration determines the represented
class in homology, we obtain for~$\delta \rightarrow 0$ that
\[ \frac{\vol(N)}{\vol(M)} \cdot \qsv  M \leq \qsv N. 
\]
By symmetry, this proves proportionality for closed hyperbolic
manifolds.

In particular, for every~$n \in \N$ there is a constant~$C_n \in
\R_{>0}$ such that all oriented closed connected hyperbolic
$n$-manifolds~$M$ satisfy
\[ \qsv M = C_n \cdot \sv M.
\]
In view of Corollary~\ref{cor:hyp3} we have~$C_3 = 1/5$ and from
Theorems~\ref{thm:hypupper} and~\ref{thm:hyplower} we obtain
\[ \frac{w^\square_n}{v_n^\triangle} \geq C_n \geq \frac{v_n^\square}{v_n^\triangle}.
\]

\begin{question}
  What are the exact values of the factor~$C_n$ for~$n\geq 4$? Do we
  have $\qsv M = C_n \cdot \sv M$ for \emph{all} oriented closed
  connected $n$-manifolds or does the geometry of the manifolds affect
  the constant?
\end{question}

\section{Gluings}\label{sec:glue}

In this section, we will prove (sub-)additivity of generalised simplicial
volumes under suitable gluings, provided that the normed model in question 
is sufficiently geometric.

\begin{thm}\label{thm:glue}
  Let $(F,\varphi,\psi)$ be a geometric normed model of the singular
  chain complex and let $N$ be an oriented compact $n$-manifold with
  UBC boundary (e.g., all components are tori). Moreover, let
  $\partial N = B_+ \sqcup B_-$ be a decomposition of the boundary
  (into possibly disconnected pieces) and let $f \colon B_-
  \longrightarrow B_+$ be an orientation reversing homeomorphism. Then
  the glued $n$-manifold
  \[ M := N \bigm/ (B_- \cong_f B_+) 
  \]
  is closed, inherits an orientation of~$N$, and satisfies
  \[ \gsv M ^F \leq \gsvlf {N^\circ}^F. 
  \]
\end{thm}

We will explain the notion of geometric normed models in
Section~\ref{subsec:geometricmodels}. In
Section~\ref{subsec:smallboundaries} we will prove that there are
$F$-fundamental cycles of~$N$ whose boundaries have small $F$-norm if
$\gsvlf{N^\circ}^F < \infty$.  Using the uniform boundary condition,
we will then be able to glue these boundaries with small chains
(Section~\ref{subsec:gluingboundaries}).

\subsection{Geometric normed models}\label{subsec:geometricmodels}

\begin{defi}[geometric normed model]
  A functorial normed chain complex~$F \colon \Top \longrightarrow
  \nCh$ is \emph{geometric} if it satisfies the following conditions 
  for all~$n \in \N$ and all spaces~$X$:
  \begin{itemize}
  \item \emph{Compact support.}  For all chains~$c \in F_n(X)$ there exist~$K
    \in K(X)$ and~$z \in F_n(K)$ such that $F_n(K \hookrightarrow
    X)(z) = c$ and $\gclone z^F \leq \gclone c^F$.
  \item \emph{$\pi_0$-Additivity.} 
    If $X = X_1 \sqcup X_2$, then for all chains~$c \in F_n(X)$ there exist~$z_1 \in F_n(X_1)$ 
    and $z_2 \in F_n(X_2)$ such that 
    \begin{align*} 
      c & = F_n(X_1 \hookrightarrow X)(z_1) + F_n(X_2 \hookrightarrow X)(z_2)  
      \\
      \gclone c ^F & = \gclone{z_1}^F + \gclone{z_2}^F.
    \end{align*}
  \item \emph{Faithfulness.} If $A \subset X$, then 
    $F_n(A \hookrightarrow X) \colon F_n(A) \longrightarrow F_n(X)$ 
    is injective.
  \end{itemize}
  
  A normed model~$(F,\varphi,\psi)$ of~$C_*(\args;\R)$ is \emph{geometric} if 
  $F$ is geometric.
\end{defi}

In view of faithfulness, for geometric normed models, we will usually 
omit the explicit notation of homomorphisms induced by inclusions of
subspaces. The faithfulness condition is added for convenience; it
could be replaced with weaker conditions (which, however, would lead
to much more cumbersome notation).

Any faithful normed model on path-connected compact spaces can be extended 
to a geometric model on all spaces by taking $\ell^1$-sums of path-connected 
components and then colimits over compact subspaces.

\begin{prop}\label{prop:classicalgeometric}
  \hfil
  \begin{enumerate}
    \item The singular chain complex~$C_*(\args;\R)$ is geometric.
    \item The cubical singular chain complex~$C_*^\square(\args;\R)$ is geometric.
  \end{enumerate}
\end{prop}
\begin{proof}
  This easily follows from the construction of the singular and the cubical singular 
  chain complex as well as the $\ell^1$-norm.
\end{proof}

In particular, Theorem~\ref{thm:glue} will hence apply to the locally
finite versions of classical simplicial volume and cubical simplicial
volume.

\subsection{Small boundaries}\label{subsec:smallboundaries}

\begin{prop}\label{prop:smallboundaries}
  Let $(F,\varphi,\psi)$ be a geometric normed model of the singular
  chain complex and let $N$ be an oriented compact 
  $n$-manifold with $\gsvlf {N^\circ}^F < \infty$. Then for
  every~$\varepsilon \in \R_{>0}$ there is a family~$(c_t)_{t \in \R_{>0}}$ 
  in~$F_n(N)$ with the following properties:
  \begin{enumerate}
    \item For all~$t \in \R_{>0}$ we have~$\partial c_t \in
      F_{n-1}(\partial N)$ and $c_t$ represents a relative
      $F$-fundamental cycle of~$N$ in~$F_n(N,\partial N)$.
    \item The family approximates the locally finite $F$-simplicial
      volume of~$N^\circ$ via
      \[ \lim_{t \rightarrow \infty} \gclone{c_t}^F 
         \leq \gsvlf {N^\circ}^F + \varepsilon. 
      \]
    \item The family has small boundaries in the sense that
      \[ \lim_{t \rightarrow \infty} \gclone{\partial c_t}^F = 0. 
      \]
  \end{enumerate}
\end{prop}

The proof is a straightforward adaption of the corresponding argument for
ordinary simplicial volume~\cite[p.~17, Chapter~6]{vbc,loehphd}.

We first introduce some notation: Let $N$ be a compact manifold. Then
$N^\circ = N \setminus \partial N$ is
homeomorphic~\cite{brownflat,connelly} to the stretched manifold
(Figure~\ref{fig:stretch})
\[ N(\infty) := N \cup_{\partial N} (\partial N \times [0,\infty)). 
\]
For~$t \in [0,\infty)$ we write
\[ N(t) := N \cup_{\partial N} (\partial N \times [0,t]) \subset N(\infty), 
\]
which is homeomorphic (relative to the boundary) to~$N$. 
Furthermore, the homotopy equivalences that collapse the
cylinder~$\partial N \times [0,\infty)$ to~$\partial N \times \{0\}$
  are denoted by
\[ p_t \colon \bigl(N(\infty), N(\infty) \setminus N(t)\bigr) 
       \longrightarrow (N,\partial N). 
\]
Clearly, the family~$(N(t))_{t \in \R_{>0}}$ is cofinal in the
directed set~$K(N(\infty))$.

\begin{figure}
  \begin{center}
    \def\torushole{%
      \draw[black!50] (0.625,0.3) arc (360:180:0.625 and 0.375);
      \draw[black!50] (0.35,0) arc (0:180:0.35 and 0.1825);
    }
    \def\surfacearc{%
              \draw (-1,0)   .. controls +(90:0.6) and +(180:0.3) 
              .. (0,0.75) .. controls +(0:0.4)  and +(180:0.4)
              .. (1.2,0.5);
              \draw (-1,0)   .. controls +(-90:0.6) and +(180:0.3) 
              .. (0,-0.75) .. controls +(0:0.4)  and +(180:0.4)
              .. (1.2,-0.5);
            }
    \def\surfaceboundary{%
      \draw (1.2,0) circle (0.12 and 0.5);
    }
    \def\mfdn{%
      \torushole
      \surfacearc
      \surfaceboundary
      \draw (0,-1.1) node {$N$};
      \draw (1.2,-1.1) node {$\partial N$};
    }
    \begin{tikzpicture}[thick]
      \mfdn
      \draw (3,0) node {$\rightsquigarrow$};
      \begin{scope}[shift={(6,0)}]
        \mfdn
        \draw (1.2,0.5) -- (3.2,0.5);
        \draw (1.2,-0.5) -- (3.2,-0.5);
        \draw (1.2,1.1) node {$0$};
        \draw (3.2,1.1) node {$\infty$};
        \begin{scope}[dashed, shift={(1,0)},black!50]
          \surfaceboundary
          \draw (1.2,1.1) node {$t$};
          \draw (1.3,-1.1) node {$\partial N(t)$};
        \end{scope}
        \begin{scope}[dashed, shift={(2,0)}]
          \surfaceboundary
        \end{scope}
        \draw (4.2,0) node {$N(\infty)$};
      \end{scope}
    \end{tikzpicture}
  \end{center}

  \caption{Stretching a manifold with boundary}
  \label{fig:stretch}
\end{figure}
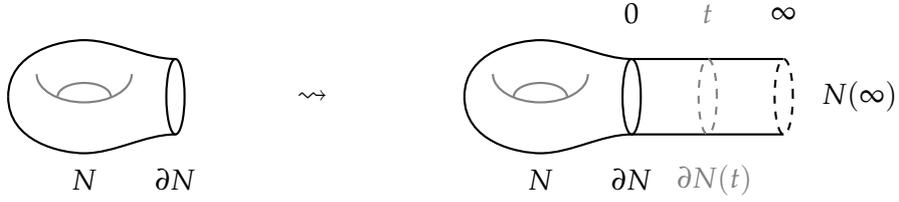

\begin{proof}[Proof of Proposition~\ref{prop:smallboundaries}]
  Because $N^\circ$ is homeomorphic to~$N(\infty)$ we also have $\gsvlf{N(\infty)}^F < \infty$. 
  Let $c \in F^{\mathrm{lf}}_n(N(\infty))$ be a locally finite
  $F$-fundamental cycle of~$N(\infty)$ with $\gclone c ^F \leq
  \gsvlf{N(\infty)}^F + 1/2 \cdot \varepsilon < \infty$. In particular, $A(c) \neq
  \emptyset$ and there exists~$\overline c \in A(c)$ with 
  \[ 
     \lim_{t \rightarrow \infty} \gclone{\overline c_{N(t)}}^F 
     \leq \gclone c ^F + \frac12\cdot\varepsilon
     \leq \gsvlf {N(\infty)}^F + \varepsilon.
  \]
  For~$t \in \R_{>0}$ we now set
  \[ c_t := F_n(p_t)(\overline c_{N(t)}) \in F_n(N). 
  \]
  By construction, $c_t$ is a chain with~$\partial c_t \in
  F_{n-1}(\partial N)$ that represents a relative $F$-fundamental
  cycle of~$(N,\partial N)$ in~$F_n(N,\partial N)$. Moreover,
  $\gclone{c_t}^F \leq \gclone{\overline c_{N(t)}}^F$, and so
  $(c_t)_{t \in \R_{>0}}$ satisfies also the second claim.

  We now prove the third claim for this family: By definition
  of~$A(c)$, the family~$(\overline c_{N(t)})_{t \in \R_{>0}}$ is
  $\gclone\cdot^F$-Cauchy. Because $F$ has compact supports, there is
  a~$t' \in \R_{>t + 1}$ such that $\overline c_{N(t)} \in
  F_n(N(t'-1))$. By construction,
  \begin{align*}
    \partial (\overline c_{N(t')}) & \in F_{n-1}\bigl(N(\infty) \setminus N(t')\bigr), 
    \\
    \partial (\overline c_{N(t)}) & \in F_{n-1}\bigl(N(t'-1)\bigr), 
    \\
    \partial( \overline c_{N(t')} - \overline c_{N(t)})
     = \partial \overline c_{N(t')} - \partial \overline c_{N(t)}
     & \in F_{n-1}\bigl(N(\infty) \setminus N(t') \sqcup N(t'-1)\bigr).
  \end{align*}
  Therefore, faithfulness and $\pi_0$-additivity of~$F$ show that
  \[      \gclone{\partial c_t}^F
     \leq \gclone{\partial(\overline c_{N(t)})}^F 
     \leq \gclone{\partial( \overline c_{N(t')} - \overline c_{N(t)})}^F 
     \leq \|\partial_n^F\| \cdot \gclone{\overline c_{N(t')} - \overline c_{N(t)}}^F.
  \]
  Because of the Cauchy condition, the last term tends to~$0$ for~$t
  \rightarrow \infty$.
\end{proof}

\subsection{Gluing along UBC boundaries}\label{subsec:gluingboundaries}

We will now glue the small boundaries provided by
Proposition~\ref{prop:smallboundaries} to obtain efficient fundamental
cycles of the glued manifold.

\begin{proof}[Proof of Theorem~\ref{thm:glue}]
  If $\gsvlf {N^\circ}^F = \infty$, then there is nothing to prove. We
  will hence assume that $\gsvlf {N^\circ}^F$ is finite. Let
  $\varepsilon > 0$ and let $\kappa \in \R_{>0}$ be a common 
  $(n-1)$-UBC constant for all boundary components of~$N$ with respect
  to~$F$. Because $F$ is geometric, we can choose a family~$(c_t)_{t
    \in \R_{>0}}$ as in Proposition~\ref{prop:smallboundaries}.

  Let $t \in \R_{>0}$. Then $\partial c_t \in F_{n-1}(\partial N)$ is 
  an $F$-fundamental cycle of~$\partial N = B_+ \sqcup B_-$. Because $F$ 
  is assumed to be faithful and $\pi_0$-additive, we can split
  \[ \partial c_t = b_{t,+} + b_{t,-} 
  \]
  into $F$-fundamental cycles~$b_{t,+} \in F_{n-1}(B_+)$, $b_{t,-} \in
  F_{n-1}(B_-)$ with
  \[ \gclone{b_{t,+}}^F + \gclone{b_{t,-}}^F = \gclone {\partial c_t}^F. 
  \]
  Moreover, because $f \colon B_- \longrightarrow B_+$ is an
  orientation reversing homeomorphism, the cycle~$w_t := b_{t,+} + F_{n-1}(f)
  (b_{t,-}) \in F_{n-1}(B_+)$ is a boundary. Therefore, the 
  uniform boundary condition guarantees the existence of 
  a chain~$b_t \in F_{n}(B_+)$ satisfying
  \[ \partial b_t = w_t
     \quad\text{and}\quad
     \gclone{b_t}^F \leq \kappa \cdot \gclone{w_t}^F 
                   \leq \kappa \cdot \gclone{\partial c_t}^F.
  \]

  We now consider the gluing projection~$\pi \colon N \longrightarrow
  M$ and the chain
  \[ z_t := F_n(\pi) (c_t - b_t) \in F_n(M). 
  \]
  By construction, $\partial z_t = 0$ and, as can be easily
  checked locally at points in~$N^\circ$, the cycle~$z_t$ is an 
  $F$-fundamental cycle of~$M$. Moreover, 
  \[ \gclone{z_t}^F 
                   \leq \gclone{c_t}^F + \kappa \cdot \gclone{\partial c_t}^F,
  \]
  and thus the properties from Proposition~\ref{prop:smallboundaries} imply 
  \[ \gsv M ^F \leq 
     \liminf_{t\rightarrow\infty} \gclone{z_t}^F
     \leq \gsvlf{N^\circ}^F + \varepsilon + \kappa \cdot 0.
  \]
  Taking $\varepsilon \rightarrow 0$ gives the desired estimate~$\gsv M ^F \leq \gsvlf{N^\circ}^F$.
\end{proof}

Alternatively, one could also try to translate the equivalence
theorem~\cite{vbc,BBFIPP} for weighted semi-norms to the setting of
normed models. However, we prefer the argument above because it is
more direct and more geometric.

\section{Cubical simplicial volume of $3$-manifolds}\label{sec:proof}

In this section, we complete the proof of Theorem~\ref{mainthm},
using the strategy described in the introduction. We review the
decomposition of $3$-manifolds in Section~\ref{subsec:3mfddecomp} and 
then prove Theorem~\ref{mainthm} in Section~\ref{subsec:proof}. 

\subsection{Decomposition of $3$-manifolds}\label{subsec:3mfddecomp}

We recall the Geometrization Theorem describing the decomposition 
of a $3$-manifold in hyperbolic and Seifert pieces.

\begin{defi}[irreducible $3$-manifold]
A $3$-manifold $M$ is called \emph{irreducible} if every embedded $2$-sphere in $M$ bounds an embedded $3$-ball in $M$.
\end{defi}
\begin{rem}
It is straightforward that an orientable irreducible $3$-manifold cannot be decomposed as a non-trivial connected sum of two manifolds.
\end{rem}
\begin{defi}[incompressible surface]\label{incompressible}
Let $S$ be a connected orientable surface different from a sphere or a disk that is properly embedded into a compact orientable $3$-manifold $M$. Then $S$ is
\emph{incompressible} if the map
$$i_\ast\colon \pi_1(S,x)\rightarrow \pi_1(M,x),$$ induced by the inclusion, is injective. Here, \emph{properly} embedded surface means that $\partial S=S\cap \partial M$. 
\end{defi}


\begin{thm}[Geometrization Theorem]\label{thm:geom}
Let $M$ be a compact orientable irreducible $3$-manifold with empty or toroidal boundary. There exists a (possibly empty) collection of disjointly embedded
incompressible tori $T_1,\dots, T_m$ in $M$ such that each component of $M$
cut along $T_1\cup\dots\cup T_m$ is hyperbolic or Seifert fibred. Furthermore, 
any such collection of tori with a minimal number of components is unique up to isotopy.
\end{thm}
For historical background and detailed references of this statement we refer to the literature~\cite[Chapter 1.7]{AFW}. 

\subsection{Generalised simplicial volume of closed $3$-manifolds and proof of Theorem~\ref{mainthm}}\label{subsec:proof}

We first formulate and prove a slight generalisation of
Theorem~\ref{mainthm} in the context of normed
models. Theorem~\ref{mainthm} will then be a special case. 

\begin{thm}\label{mainthmgen}
  Let $(F,\varphi,\psi)$ be a geometric normed model of~$C_*(\args;\R)$ and let 
  \begin{align*} 
    C_F := \sup \Bigl\{ \frac{v_3^\triangle \cdot \gsvlf N ^F}{ \vol(N)} 
              \Bigm| 
              & \;\text{$N$ is a complete oriented connected hyperbolic $3$-manifold}
              \\
              & \;\text{of finite volume}
              \Bigr\}.
  \end{align*}
  Then $C_F \leq \|\psi_3^{\Delta^3}\|$ and for all oriented closed
  $3$-manifolds~$M$ we have 
  \[ \gsv M^F \leq C_F \cdot \sv M.
  \]
\end{thm}
\begin{proof}
  Because $\svlf M = \vol(M)/v_3^\triangle$ holds for all complete hyperbolic
  $3$-manifolds of finite volume (Theorem~\ref{thm:svhyp}), we
  obtain~$C_F \leq \|\psi_3^{\Delta^3}\|$ from
  Proposition~\ref{prop:equivalence}. 

  By the Geometrization Theorem~\ref{thm:geom}, there exist $m, n \in
  \N$ and disjointly embedded incompressible tori~$T_1,\dots,T_m$ in~$M$
  such that each of the components~$N_1,\dots, N_n$ obtained by
  cutting~$M$ along~$T_1 \cup \dots \cup T_m$ is Seifert fibred or
  admits a complete finite volume hyperbolic structure on its interior. 
  
  Let $S \subset \{1,\dots,n\}$ be the set of indices belonging to the
  Seifert fibred pieces (whence the indices in~$H := \{1,\dots,n\}
  \setminus S$ belong to the hyperbolic pieces).  By
  Corollary~\ref{cor:seifert}, the definition of~$C_F$, and
  Theorem~\ref{thm:svhyp} we obtain
  \[ \gsvlf{N_j^\circ}^F \leq 
  \begin{cases}
     0 & \text{if $j \in S$}\\
     C_F \cdot \frac{\vol(N_j^\circ)}{v_3^\triangle} & \text{if $j \in H$}
   \end{cases}
  \]
  for all~$j \in \{1,\dots,n\}$. All these values are finite.  
  Therefore, we can apply Theorem~\ref{thm:glue} and obtain
  \begin{align*}
    \gsv M ^F& 
    \leq \gsvlf{N_1^\circ \sqcup \dots \sqcup N_n^\circ}^F
    = \sum_{j=1}^n \gsvlf{N_j^\circ}^F 
    \leq C_F \cdot \sum_{j \in H} \frac{\vol(N_j^\circ)}{v_3^\triangle}.
  \end{align*}
  The last sum is equal to~$\sv M$ by Theorem~\ref{thm:sv3}. Thus, 
  $\gsv M^F \leq C_F \cdot \sv M$, as claimed.
\end{proof}

As a special case, we obtain Theorem~\ref{mainthm}: 

\begin{proof}[Proof of Theorem~\ref{mainthm}]
  Let $M$ be an oriented closed $3$-manifold. In view of the lower
  bound established in Corollary~\ref{cor:3lowerbound}, it suffices to
  prove the upper bound~$\qsv M \leq 1/5 \cdot \sv M$.

  By Proposition~\ref{prop:classicalgeometric}, all cubical
  models~$(F,\varphi,\psi)$ are geometric, and so
  Theorem~\ref{mainthmgen} applies to cubical simplicial volume. On
  the other hand, Corollary~\ref{cor:hyp3} tells us that
  $C_F = 1/5$. Therefore, 
  $\qsv M \leq 1/5 \cdot \sv M$.
\end{proof}


\end{document}